\documentclass[12pt]{article}

\usepackage{graphicx}
\usepackage{amssymb}
\usepackage{amsthm}
\usepackage{amsmath}
\usepackage{amsfonts}

\newcommand{\mscomm}[1]{\begingroup #1\endgroup}

\usepackage[breaklinks,pdfstartview=FitH,colorlinks,linktocpage,bookmarks=true]{hyperref}
\usepackage{geometry}
\usepackage{pgf,tikz}\usetikzlibrary{arrows}

\newtheorem{theorem}{Theorem}[section]
\newtheorem{lemma}[theorem]{Lemma}

\newtheorem{corollary}[theorem]{Corollary}
\newtheorem{projection-lemma}[theorem]{Projection Lemma}
\newtheorem{gradient-approximation-lemma}[theorem]{Gradient Approximation Lemma}
\newtheorem{energy-approximation-lemma}[theorem]{Energy Approximation Lemma}
\newtheorem{energy-conservation-principle}[theorem]{Energy Conservation Principle}
\newtheorem{energy-positivity-principle}[theorem]{Energy Positivity Principle}
\newtheorem{derivative-estimation-lemma}[theorem]{Derivative Estimation Lemma}
\newtheorem{Riemann-bilinear-identity}[theorem]{Riemann Bilinear Identity}
\newtheorem{stokes-formula}[theorem]{Stokes' Formula}
\newtheorem{First-Existence-and-Uniqueness-Theorem}[theorem]{First Existence and Uniqueness Theorem}
\newtheorem{Second-Existence-and-Uniqueness-Theorem}[theorem]{Second Existence \& Uniqueness Theorem}
\newtheorem{Convergence-Theorem-for-Abelian-integrals}[theorem]{Convergence Theorem for Abelian Integrals}
\newtheorem{Convergence-Theorem-for-Period-Matrices}[theorem]{Convergence Theorem for Period Matrices}
\newtheorem{variational-principle}[theorem]{Variational Principle}
\newtheorem{conjugate-functions-principle}[theorem]{Conjugate Functions Principle}
\newtheorem{energy-convergence-lemma}[theorem]{Energy Form Convergence Lemma}
\newtheorem{energy-matrix-lemma}[theorem]{Energy Matrix Lemma}
\newtheorem{equicontinuity-lemma}[theorem]{Equicontinuity
Lemma}
\newtheorem{Riemann-Roch-theorem}[theorem]{Riemann--Roch Theorem}
\newtheorem{period-matrix-lemma}[theorem]{Period Matrix Lemma}
\newtheorem{3rd-existence-and-uniqueness-theorem}[theorem]{Third Existence and Uniqueness Theorem}
\newtheorem{4th-existence-and-uniqueness-theorem}[theorem]{Fourth Existence and Uniqueness Theorem}
\newtheorem{restriction-energy-convergence-lemma}[theorem]{Energy Convergence Lemma}
\newtheorem{interpolation-lemma}[theorem]{Interpolation Lemma}
\newtheorem{convergence-multi-valued-discrete-harmonic- functions}
[theorem]{Convergence Theorem for multi-valued discrete harmonic functions}

\theoremstyle{definition}

\newtheorem{example}[theorem]{Example}
\newtheorem{problem}[theorem]{Problem}

\theoremstyle{remark}
\newtheorem{remark}[theorem]{Remark}

\DeclareMathOperator{\Real}{Re}
\DeclareMathOperator{\Imaginary}{Im}

\begin{document}




\title{Discrete Riemann surfaces:\\ linear discretization and its convergence}

\author{Alexander Bobenko and Mikhail Skopenkov}

\date{}
\maketitle




\begin{abstract}
We develop linear discretization of complex analysis, originally
introduced by R. Isaacs, J. Ferrand, R. Duffin, and C. Mercat. We prove 
convergence of discrete period matrices and 
discrete Abelian integrals to their continuous
counterparts. 
We also prove a discrete counterpart of the Riemann--Roch theorem.
The proofs use energy estimates inspired by electrical
networks.
%

\smallskip


\end{abstract}



\footnotetext[0]{
The first author was partially supported  by SFB/TR 109 ``Discretization in geometry and dynamics''.  The second author was partially supported by the President of the Russian Federation grant MK-3965.2012.1, 
by ``Dynasty'' foundation, and by the Simons--IUM fellowship.
}

\section{Introduction}\label{sec-intro}

The history of linear discretizations of harmonic and holomorphic
functions can be traced back to the early finite-element
literature~\cite{A1-BOB-CFL_1928, A1-BOB-duffin_1953}. The
discretization of the Dirichlet variational principle leads to linear
discretizations of the Laplace equation and of the Cauchy--Riemann
equations on a square grid. Lelong-Ferrand developed the theory of
discrete harmonic and holomorphic functions on a square grid to a
level that allowed her to prove the Riemann
mapping theorem~\cite{A1-BOB-ferrand_1944, A1-BOB-lelong-ferrand_1955}. Duffin 
took first steps to extend the theory beyond the square grid. He
considered arbitrary triangulations of planar domains and discovered the now famous cotan-weights~\cite{A1-BOB-duffin_1959}.
The cotan-Laplace-operator turned out to be important in discrete differential geometry of surfaces~\cite{MR1246481} and in computer graphics~\cite{MR2047000}.

A striking feature of some $2$-dimensional discrete models in
statistical physics is that they exhibit conformally invariant
properties in the scaling limit:  
for site percolation
on a triangular grid~\cite{A1-BOB-smirnov_2001},
for the random
cluster model~\cite{A1-BOB-smirnov_2010}, 
for
the Ising model~\cite{Chelkak-Smirnov-11}, 
for domino tiling
~\cite{A1-BOB-Kenyon_2000}. In all cases, the linear theory of
discrete analytic functions on regular grids was instrumental. 
  

Mercat generalized the linear theory from planar domains to discrete
Riemann surfaces~\cite{A1-BOB-Mercat_2001} 
and in particular introduced discrete period matrices~\cite{MR2349680}. 
Numerical experiments to compute period matrices for triangulated surfaces and compare them with known period matrices for the smooth surfaces were performed in~\cite{A1-BOB-BobenkoMercatSchmies}. First attempts to prove the convergence were made in~\cite{Mercat-02} but it
remained an open problem; see Remark~\ref{rem-error}.

A very different ``nonlocal'' linear theory for discrete holomorphic functions on triangulated surfaces was introduced by Wilson \cite{Wilson-08}. For this discretization the 
convergence of period matrices to their continuous counterparts is also discussed in \cite{Wilson-08}. 
Yet another linear theory for discrete holomorphic functions on
regular triangle lattices, due to 
Dynnikov--Novikov~\cite{A1-BOB-DynnikovNovikov}, was motivated by the theory
of integrable systems. A discretization of the Riemann--Roch theorem for graphs was introduced by Baker--Norine~\cite{MR2355607}.
 
Important nonlinear discrete models of complex analysis and Riemann
surfaces involve circle packings or, more generally, patterns of
intersecting or disjoint circles~\cite{A1-BOB-thurston_notes, A1-BOB-rodin_sullivan_1987, A1-BOB-Stephenson_Book}. 
The linear theory of discrete holomorphic functions on square grid
can be obtained by linearization from the nonlinear circle pattern
theory. 
Discrete holomorphic functions describe infinitesimal
deformations of circle patterns~\cite{A1-BOB-BobenkoMercatSuris}. 
A strikingly simple concept of discrete conformal equivalence for
triangle meshes leads to a new nonlinear model of discrete Riemann
surfaces~\cite{A1-BOB-SpringbornSchroederPinkall},
\cite{A1-BOB-BobenkoPinkallSpringborn}. 

The problem of convergence of discrete analytic functions to their continuous counterparts is certainly one of the most important issues. Convergence for 
the linear theory 
was proved by Courant--Friedrichs--Lewy \cite{A1-BOB-CFL_1928}
for square lattices, 
by Chelkak--Smirnov \cite{MR2824564}
for rhombic lattices, and
by one of the authors~\cite{A1-Skopenkov} for orthogonal quadrilateral lattices. Although there are some striking convergence
results~\cite{A1-BOB-rodin_sullivan_1987, A1-BOB-schramm_1997,A1-BOB-he_schramm_1996}, see also \cite{A1-BOB-Buecking},
the situation is far from satisfying. 

In this paper we prove convergence of discrete period
matrices (Theorem~\ref{th-main}) and discrete Abelian integrals (Theorem~\ref{th-Convergence of Abelian integrals}) to their continuous counterparts. In particular, we obtain an error estimate~\eqref{eq-error-estimate} for period matrices. 
We also prove a discrete analogue of the Riemann--Roch theorem (Theorem~\ref{Riemann-Roch-theorem}). 

The proofs use energy estimates inspired by electrical networks, which is a highly nontrivial extension of~\cite{A1-Skopenkov} to the case when the curvature is present. Energy estimates allow to prove convergence of discrete period matrices directly, without proving convergence of Abelian integrals first (as in previous approaches). An elementary introduction to our method
can be found in \cite{Dorichenko-etal-11, Skopenkov-etal-12}.

Main results of the paper are stated in Section~\ref{sec-main} and proved in Sections~\ref{sec-period}--\ref{sec-integrals}. 
Discrete Riemann--Roch theorem is stated and proved
in Section~\ref{sec-riemann-roch}. A generalization of our setup, \mscomm{results of numerical experiments}, and some open problems are given in Section~\ref{sec-generalization}. 

\section{Main ideas}\label{sec-main}

\subsection{Discrete harmonic and discrete analytic functions}
\label{ssec:defs}

Let ${{S}}$ be a \emph{polyhedral surface}, i.e., an oriented $2$-dimensional manifold without boundary equipped with a piecewise flat metric having isolated conical singularities. \mscomm{An example of a polyhedral surface is the surface of a polyhedron in 3-dimensional space.}
Let ${{T}}$ be a geodesic triangulation of the polyhedral surface ${{S}}$ such that \mscomm{all faces are flat triangles; in particular,}
all the singular points of the metric are vertices of ${{T}}$.
Denote by ${{T}}^0$, ${{T}}^1$, $\vec{{{T}}}^1$, ${{T}}^2$ the sets of vertices, edges, oriented edges, faces, respectively.
Introduce \emph{cotan edge weights} by the formula
$$
\mscomm{{c}}(e)=\frac{1}{2}\cot\alpha_e+\frac{1}{2}\cot\beta_e,
$$
where $\alpha_e$ and $\beta_e$ are the angles opposite to an edge $e\in {{T}}^1$ in the $2$ triangles sharing 
$e$;
see Figure~\ref{fig:2triangles}.

\begin{figure}[htbp]
\centerline{\small
\definecolor{zzttqq}{rgb}{0.6,0.2,0}
\definecolor{ffqqqq}{rgb}{1,0,0}
\definecolor{qqqqff}{rgb}{0,0,1}
\begin{tikzpicture}[line cap=round,line join=round,>=triangle 45,x=0.5cm,y=0.5cm]
\clip(-3.2,-3.34) rectangle (3.5,6.1);
\draw [shift={(1,-3)},color=zzttqq,fill=zzttqq,fill opacity=0.1] (0,0) -- (63.43:0.6) arc (63.43:135:0.6) -- cycle;
\draw [shift={(2,6)},color=zzttqq,fill=zzttqq,fill opacity=0.1] (0,0) -- (-135:0.6) arc (-135:-78.69:0.6) -- cycle;
\draw [line width=1.2pt,color=qqqqff] (-3,1)-- (3,1);
\draw [line width=1.2pt,color=qqqqff] (3,1)-- (2,6);
\draw [line width=1.2pt,color=qqqqff] (2,6)-- (-3,1);
\draw [line width=1.2pt,color=qqqqff] (3,1)-- (1,-3);
\draw [line width=1.2pt,color=qqqqff] (1,-3)-- (-3,1);
\draw [line width=1.2pt,color=qqqqff] (-3,1)-- (3,1);
\draw (0,3)-- (2,6);
\draw (0.85,4.6) -- (1.15,4.4);
\draw (0,3)-- (3,1);
\draw (1.6,2.15) -- (1.4,1.85);
\draw (0,3)-- (-3,1);
\draw (-1.4,1.85) -- (-1.6,2.15);
\draw (3,1)-- (0,0);
\draw (1.56,0.43) -- (1.5,0.6);
\draw (1.5,0.4) -- (1.44,0.57);
\draw (0,0)-- (-3,1);
\draw (-1.5,0.4) -- (-1.44,0.57);
\draw (-1.56,0.43) -- (-1.5,0.6);
\draw (0,0)-- (1,-3);
\draw (0.57,-1.44) -- (0.4,-1.5);
\draw (0.6,-1.5) -- (0.43,-1.56);
\draw [shift={(1,-3)},color=zzttqq] (63.43:0.6) arc (63.43:135:0.6);
\draw [shift={(1,-3)},color=zzttqq] (63.43:0.5) arc (63.43:135:0.5);
\draw [->,color=qqqqff] (-3,1) -- (3,1);
\fill [color=qqqqff] (-3,1) circle (1.5pt);
\draw[color=qqqqff] (-2.8,0.22) node {$t_e$};
\fill [color=qqqqff] (3,1) circle (1.5pt);
\draw[color=qqqqff] (3.2,0.16) node {$h_e$};
\fill [color=qqqqff] (2,6) circle (1.5pt);
\draw[color=qqqqff] (-0.76,4.12) node {${T}$};
\fill [color=qqqqff] (1,-3) circle (1.5pt);
\fill [color=ffqqqq] (0,3) ++(-1.5pt,0 pt) -- ++(1.5pt,1.5pt)--++(1.5pt,-1.5pt)--++(-1.5pt,-1.5pt)--++(-1.5pt,1.5pt);
\draw[color=ffqqqq] (0.2,2.4) node {$l_e$};
\draw[color=black] (1.82,2.48) node {$Q$};
\fill [color=ffqqqq] (0,0) ++(-1.5pt,0 pt) -- ++(1.5pt,1.5pt)--++(1.5pt,-1.5pt)--++(-1.5pt,-1.5pt)--++(-1.5pt,1.5pt);
\draw[color=ffqqqq] (0.28,0.5) node {$r_e$};
\draw[color=zzttqq] (1.12,-1.82) node {$\beta_e$};
\draw[color=zzttqq] (1.8,4.76) node {$\alpha_e$};
\draw[color=qqqqff] (0.2,1.24) node {$e$};
\end{tikzpicture}
}
\caption{Notation associated with an edge $e\in{{T}}^1$; see Section~\ref{ssec:defs} for explanation.}
\label{fig:2triangles}
\end{figure}
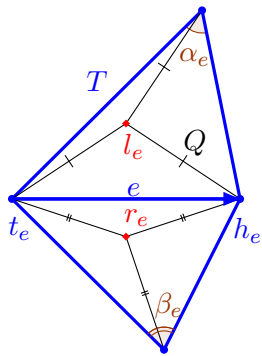

A function $u\colon {{T}}^0\to\mathbb{R}$ is \emph{discrete harmonic}, if for each vertex $z\in {{T}}^0$ we have 
\begin{equation}\label{eq-def-harmonic}
\sum_{xy\in {{T}}^1\,:\,x=z}\mscomm{{c}}(xy)(u(x)-u(y))=0.
\end{equation}


\begin{example}\label{ex-grid} (See Figure~\ref{figure:torus} to the left) Let \mscomm{$\widetilde{{T}}$} be an infinite square grid \mscomm{in the plane} with the SW--NE diagonal drawn in each square. Then the sides of the  squares have unit weights and the weights of diagonals vanish. Thus a function on the vertices of \mscomm{$\widetilde{{T}}$} is discrete harmonic, if and only if the value at a vertex equals the average of the values at the neighbors (diagonal neighbors are not taken into account).
\end{example}






For an oriented 
edge $e\in\vec{{{T}}}^1$ denote by $h_e\in {{T}}^0$, $t_e\in {{T}}^0$, $l_e\in {{T}}^2$, $r_e\in {{T}}^2$ the \emph{head}, the \emph{tail}, the \emph{left shore}, the \emph{right shore} of $e$, respectively; see Figure~\ref{fig:2triangles}. Two functions $u\colon {{T}}^0\to \mathbb{R}$ and $v\colon {{T}}^2\to \mathbb{R}$
are \emph{conjugate}, if for each oriented 
edge $e\in \vec{{{T}}}^1$ we have
\begin{equation}\label{eq-def-analytic2}
v(l_e)-v(r_e)=\mscomm{{c}}(e)(u(h_e)-u(t_e)).
\end{equation}
The pair $f=(u\colon {{T}}^0\to \mathbb{R},v\colon {{T}}^2\to \mathbb{R})$ of two conjugate functions is called a \emph{discrete analytic function}; see an example in Figure~\ref{figure:torus} to the left. We write $\mathrm{Re}f:=u$ and $\mathrm{Im}f:=v$.
We write $f=\mathrm{const}$, if both $u$ and $v$ are constant functions, not necessarily equal to each other. A direct checking shows that on simply-connected surfaces ${{S}}$ discrete harmonic functions are precisely the real parts of discrete analytic functions.

\begin{figure}[htbp]
\begin{tabular}{ccccc}
\definecolor{qqqqff}{rgb}{0,0,1}
\begin{tikzpicture}[line cap=round,line join=round,>=triangle 45,x=3.6cm,y=3.6cm]
\draw [color=black,, xstep=1.8cm,ystep=1.8cm] (-3.83,1.27) grid (-2.33,2.68);
\clip(-3.83,1.27) rectangle (-2.33,2.68);
\draw [domain=-3.83:-2.33] plot(\x,{(-5.5-1*\x)/-1});
\draw [domain=-3.83:-2.33] plot(\x,{(-3-0.5*\x)/-0.5});
\draw [domain=-3.83:-2.33] plot(\x,{(-2.5-0.5*\x)/-0.5});
\draw [domain=-3.83:-2.33] plot(\x,{(-4.5-1*\x)/-1});
\draw [line width=1.6pt,domain=-3.83:-2.33] plot(\x,{(-3.25-0.5*\x)/-0.5});
\draw [domain=-3.83:-2.33] plot(\x,{(-2-0.5*\x)/-0.5});
\draw [->] (-3.42,1.42) -- (-2.92,1.42);
\draw (-2.97,2.44) node[anchor=north west] {$3/2$};
\draw (-2.77,2.24) node[anchor=north west] {$3/2$};
\draw (-3.46,2.44) node[anchor=north west] {$3/2$};
\draw (0,0.07) node[anchor=north west] {$$3/2$$};
\draw (0,0.07) node[anchor=north west] {$$3/2$$};
\draw (-3.27,2.23) node[anchor=north west] {$3/2$};
\draw (-2.97,1.94) node[anchor=north west] {$1/2$};
\draw (0,0.07) node[anchor=north west] {$$3/2$$};
\draw (-2.77,1.73) node[anchor=north west] {$1/2$};
\draw (-3.41,1.93) node[anchor=north west] {$1/2$};
\draw (-3.27,1.72) node[anchor=north west] {$1/2$};
\draw [->] (-3.6,1.6) -- (-3.6,2.1);
\fill [color=qqqqff] (-3,2.5) ++(-1.5pt,0 pt) -- ++(1.5pt,1.5pt)--++(1.5pt,-1.5pt)--++(-1.5pt,-1.5pt)--++(-1.5pt,1.5pt);
\draw[color=qqqqff] (-2.85,2.58) node {$\mathbf{1}$};
\fill [color=qqqqff] (-3.5,2.5) ++(-1.5pt,0 pt) -- ++(1.5pt,1.5pt)--++(1.5pt,-1.5pt)--++(-1.5pt,-1.5pt)--++(-1.5pt,1.5pt);
\draw[color=qqqqff] (-3.35,2.57) node {$\mathbf{0}$};
\fill [color=qqqqff] (-3,2) ++(-1.5pt,0 pt) -- ++(1.5pt,1.5pt)--++(1.5pt,-1.5pt)--++(-1.5pt,-1.5pt)--++(-1.5pt,1.5pt);
\draw[color=qqqqff] (-2.85,2.08) node {$\mathbf{1}$};
\fill [color=qqqqff] (-3.5,1.5) ++(-1.5pt,0 pt) -- ++(1.5pt,1.5pt)--++(1.5pt,-1.5pt)--++(-1.5pt,-1.5pt)--++(-1.5pt,1.5pt);
\draw[color=qqqqff] (-3.34,1.56) node {$\mathbf{0}$};
\fill [color=qqqqff] (-2.5,2) ++(-1.5pt,0 pt) -- ++(1.5pt,1.5pt)--++(1.5pt,-1.5pt)--++(-1.5pt,-1.5pt)--++(-1.5pt,1.5pt);
\draw[color=qqqqff] (-2.36,2.08) node {$\mathbf{2}$};
\fill [color=qqqqff] (-2.5,1.5) ++(-1.5pt,0 pt) -- ++(1.5pt,1.5pt)--++(1.5pt,-1.5pt)--++(-1.5pt,-1.5pt)--++(-1.5pt,1.5pt);
\draw[color=qqqqff] (-2.36,1.57) node {$\mathbf{2}$};
\draw[color=black] (-2.82,1.39) node {$d_{\alpha_1}$};
\fill [color=qqqqff] (-3.5,2) ++(-1.5pt,0 pt) -- ++(1.5pt,1.5pt)--++(1.5pt,-1.5pt)--++(-1.5pt,-1.5pt)--++(-1.5pt,1.5pt);
\draw[color=qqqqff] (-3.34,2.07) node {$\mathbf{0}$};
\fill [color=qqqqff] (-3,1.5) ++(-1.5pt,0 pt) -- ++(1.5pt,1.5pt)--++(1.5pt,-1.5pt)--++(-1.5pt,-1.5pt)--++(-1.5pt,1.5pt);
\draw[color=qqqqff] (-2.85,1.56) node {$\mathbf{1}$};
\fill [color=qqqqff] (-2.5,2.5) ++(-1.5pt,0 pt) -- ++(1.5pt,1.5pt)--++(1.5pt,-1.5pt)--++(-1.5pt,-1.5pt)--++(-1.5pt,1.5pt);
\draw[color=qqqqff] (-2.36,2.57) node {$\mathbf{2}$};
\draw[color=black] (-3.57,2.2) node {$d_{\beta_1}$};
\end{tikzpicture} & 
 & 
\begin{tikzpicture}[line cap=round,line join=round,>=triangle 45,x=3.6cm,y=3.6cm]
\clip(-3.8,1.26) rectangle (-2.73,2.69);
\draw [->] (-3.5,1.5) -- (-3.5,2);
\draw [->] (-3,1.5) -- (-3,2);
\draw [->] (-3.5,1.5) -- (-3,1.5);
\draw [->] (-3.5,2) -- (-3,2);
\draw (-3.5,1.5)-- (-3,2);
\draw[color=black] (-3.59,1.74) node {$\beta_1$};
\draw[color=black] (-2.92,1.75) node {$\beta_1$};
\draw[color=black] (-3.21,1.38) node {$\alpha_1$};
\draw[color=black] (-3.25,2.06) node {$\alpha_1$};
\end{tikzpicture} &
& 
\begin{tikzpicture}[line cap=round,line join=round,>=triangle 45,x=3.6cm,y=3.6cm]
\clip(-3.59,1.27) rectangle (-2.37,2.7);
\draw [rotate around={-0.77:(-3.01,1.75)}] (-3.01,1.75) ellipse (1.85cm and 0.9cm);
\draw [shift={(-3.02,2.16)}] plot[domain=3.94:5.48,variable=\t]({1*0.49*cos(\t r)+0*0.49*sin(\t r)},{0*0.49*cos(\t r)+1*0.49*sin(\t r)});
\draw [shift={(-3.02,1.4)}] plot[domain=0.94:2.21,variable=\t]({1*0.45*cos(\t r)+0*0.45*sin(\t r)},{0*0.45*cos(\t r)+1*0.45*sin(\t r)});
\draw [shift={(-3.18,1.6)}] plot[domain=-0.49:0.42,variable=\t]({1*0.2*cos(\t r)+0*0.2*sin(\t r)},{0*0.2*cos(\t r)+1*0.2*sin(\t r)});
\draw [shift={(-2.98,2.45)}] plot[domain=4.03:5.33,variable=\t]({1*0.84*cos(\t r)+0*0.84*sin(\t r)},{0*0.84*cos(\t r)+1*0.84*sin(\t r)});
\draw[color=black] (-2.96,1.41) node {$\beta_1$};
\draw[color=black] (-2.42,1.73) node {$\alpha_1$};
\end{tikzpicture} 
\\
$\widetilde{{{T}}}$ & 
$\overset{p}{\to}$ & 
${{{T}}}$ &
$\approx $ &
$S^1\times S^1$
\end{tabular}
\caption{The plane (left) is the universal covering of a torus (right).
The triangulation $\widetilde{{{T}}}$ covers the ``triangulation'' ${{T}}$ of the torus. \mscomm{The deck transformations of the covering are translations along the vectors $d_{\alpha_1}$ and $d_{\beta_1}$.}
Numbers represent the values of a discrete analytic function, namely,  the discrete Abelian integral $\phi^1_{{{T}}}$ of the first kind.
\mscomm{Bold} numbers near the vertices are the values of the real part $\mathrm{Re}\,\phi^1_{{{T}}}$. It is a discrete harmonic function. \mscomm{Thin} numbers  inside the faces are the values of the imaginary part $\mathrm{Im}\,\phi^1_{{{T}}}$. 
See Examples~\ref{ex-grid} and~\ref{ex-torus} for the details.}
\label{figure:torus}
\end{figure}
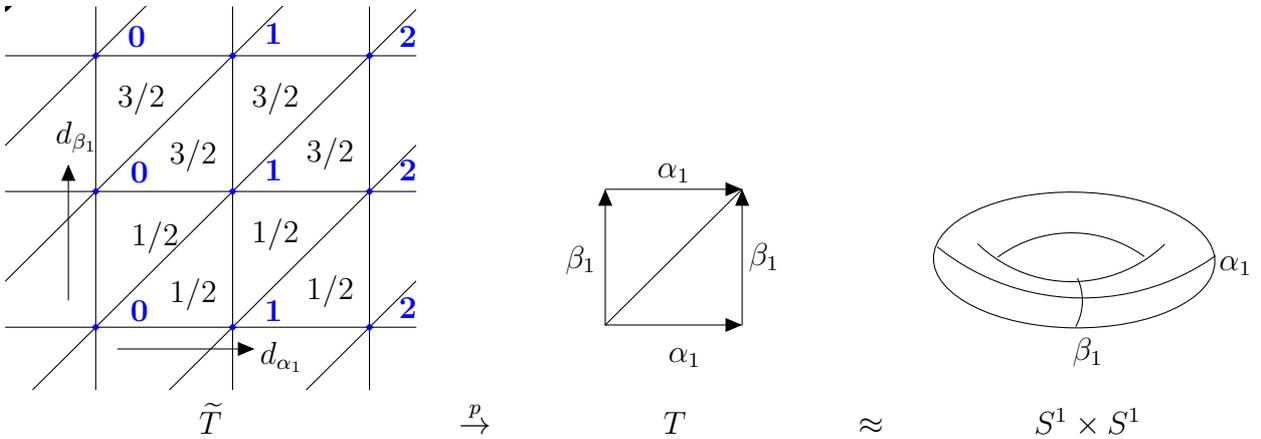

\subsection{Discrete Abelian integrals of the first kind}\label{ssec:abelian}

We are going to consider multi-valued functions on the vertices and the faces of the triangulation ${{T}}$. Informally,  a multi-valued function is the one whose value changes after performing a nontrivial loop on the surface ${{S}}$. 
For formal definition we need the following notions.

In what follows assume that the surface ${{S}}$ is closed and has genus $g\ge 1$. Denote by $p\colon\widetilde{{{S}}}\to {{S}}$ the universal covering of ${{S}}$ and by $p\colon\widetilde{{{T}}}\to {{T}}$ the induced universal covering of ${{T}}$. 
Fix \mscomm{a base point $z_0\in \widetilde{{{S}}}$ and} closed paths
$\alpha_1,\dots,\alpha_{g},\beta_1,\dots,\beta_g\colon [0,1]\to {{S}}$ forming a standard basis of the fundamental group $\pi_1({{S}}\mscomm{,pz_0})$ \mscomm{so that $\alpha_1\beta_1\alpha_1^{-1}\beta_1^{-1}\dots \alpha_g\beta_g\alpha_g^{-1}\beta_g^{-1}$ is null-homotopic.}
Each closed path 
$\alpha\colon [0,1]\to {{S}}$ with $\alpha (0)=\alpha (1)=pz_0$
determines the \emph{deck transformation} $d_{\alpha}\colon \widetilde{{{S}}}\to \widetilde{{{S}}}$, i.e., the homeomorphism such that $p\circ d_{\alpha}=p$ and $d_{\alpha}(\mscomm{z_0})=\widetilde\alpha(1)$,
where $\widetilde\alpha\colon \mathbb{R}\to \widetilde{{{S}}}$ is \mscomm{the} lift of $\alpha\colon [0,1]\to {{S}}$ \mscomm{such that $\widetilde\alpha(0)=z_0$}. The induced deck transformation of $\widetilde{{{T}}}$ is also denoted by $d_{\alpha}\colon \widetilde{{{T}}}\to \widetilde{{{T}}}$.

A \emph{multi-valued function with periods} $A_1,\dots,A_{g},B_1,\dots,B_{g}\in\mathbb{C}$ is a pair of functions $f=(\mathrm{Re}f\colon\widetilde{{{T}}}^0\to\mathbb{R},\mathrm{Im}f\colon\widetilde{{{T}}}^2\to\mathbb{R})$ such that for each $k=1,\dots,g$ and each $z\in \widetilde{{{T}}}^0$, $w\in \widetilde{{{T}}}^2$ we have
\begin{align*}
[\mathrm{Re}f](d_{\alpha_k}z)-[\mathrm{Re}f](z)&=\mathrm{Re}\, A_k; & [\mathrm{Re}f](d_{\beta_k}z)-[\mathrm{Re}f](z)&=\mathrm{Re}\, B_k;\\
[\mathrm{Im}f](d_{\alpha_k}w)-[\mathrm{Im}f](w)&=\mathrm{Im}\, A_k; &
[\mathrm{Im}f](d_{\beta_k}w)-[\mathrm{Im}f](w)&=\mathrm{Im}\, B_k.
\end{align*}
The numbers $A_1,\dots,A_{g}$ and $B_1,\dots,B_{g}$ are called the \emph{A-periods} and the \emph{B-periods} of the multi-valued function $f$, respectively.
A multi-valued discrete analytic function is called a \emph{discrete Abelian integral of the 1st kind} or \emph{discrete holomorphic integral}.

\begin{example}\label{ex-torus} (See Figure~\ref{figure:torus}.) Let ${{S}}$ be the surface obtained by gluing opposite sides of a unit square. 
Let ${{T}}$ be the ``triangulation'' obtained by drawing the SE--NW diagonal of the square. Then $\widetilde{{{S}}}$ is the \mscomm{complex plane $\mathbb{C}$} and $\widetilde{{{T}}}$ \mscomm{is the triangulation from Example~\ref{ex-grid}}.
The deck transformations $d_{\alpha_1},d_{\beta_1}\colon \widetilde{{{S}}}\to\widetilde{{{S}}}$ are unit translations along the sides of the squares. For a face $w\in \widetilde{{{T}}}^2$ denote by $w^*$ 
the circumcenter of $w$.
Then a discrete Abelian integral of the 1st kind with the period $A_1\in\mathbb{C}$ is given by the formulae 
$[\mathrm{Re}f](z):=\mathrm{Re}(A_1 z)$ for each vertex $z\in \widetilde{{{T}}}^0$ and 
$[\mathrm{Im}f](w):=\mathrm{Im}(A_1 w^*)$
for each face $w\in \widetilde{{{T}}}^2$.
Actually these formulae produce discrete Abelian integrals of the 1st kind on any triangulated flat torus ${{S}}=\mathbb{C}/(\mathbb{Z}+\mathbb{Z}\eta)$, where $\eta\in\mathbb{C}-\mathbb{R}$.
\end{example}


\subsection{Discrete period matrices}


To proceed we need the following theorem essentially due to Mercat.
It is proved in Section~\ref{ssec-exist-unique}.

\begin{First-Existence-and-Uniqueness-Theorem}\label{th-fund}
For any numbers $A_1,\dots,A_{g}\in\mathbb{C}$ there exists a discrete Abelian integral of the 1st kind with A-periods $A_1,\dots,A_{g}$.
It is unique up to constant.
\end{First-Existence-and-Uniqueness-Theorem}

For each $l=1,\dots,{g}$ denote by $\phi^l_{{{T}}}=(\mathrm{Re}\,\phi^l_{{{T}}}\colon\widetilde{{{T}}}^0\to\mathbb{R},\mathrm{Im}\,\phi^l_{{{T}}}\colon\widetilde{{{T}}}^2\to\mathbb{R})$ the unique (up to constant) discrete Abelian integral of the 1st kind with A-periods given by the formula $A_k=\delta_{kl}$, where $k=1,\dots,g$.
The $g\times g$ matrix $\Pi_{{T}}$ whose $l$-th column is formed by the B-periods of $\phi^l_{{{T}}}$, where $l=1,\dots,g$, is called the \emph{period matrix} of the triangulation~${{T}}$. 

\begin{example} Let ${{T}}$ be an arbitrary triangulation of a flat torus ${{S}}=\mathbb{C}/(\mathbb{Z}+\mathbb{Z}\eta)$ with $\eta\in\mathbb{C}-\mathbb{R}$. Then 
$\Pi_{{T}}=\eta=\Pi_{{S}}$ by the explicit construction of discrete Abelian integrals in Example~\ref{ex-torus}.
\end{example}

\subsection{Convergence of discrete period matrices}\label{ssec-convergence1}


The polyhedral surface ${{S}}$ has a natural complex structure (see Section~\ref{ssec-energy-convergence-cone}).
Indeed, identify each face $w\in\widetilde{{T}}^2$ with a triangle in the complex plane $\mathbb{C}$ by an orientation-preserving isometry. A function $f\colon \widetilde{{S}}\to \mathbb{C}$ is \emph{analytic}, if it is continuous and its restriction to the interior of each face is analytic.
Given the notion of an analytic function, a \emph{basis of Abelian integrals} $\phi^l_{{{S}}}\colon \widetilde{{{S}}}\to \mathbb{C}$ of the first kind and the \emph{period matrix} $\Pi_{{S}}$ of the surface ${{S}}$ are defined analogously to the discrete case above. 
The \emph{aperture} of a vertex $z\in {{T}}^0$ is the sum of all the face angles meeting at the vertex. Denote by $\gamma_z$ the value $2\pi$ divided by the aperture.
Denote $\gamma_{{S}}:=\min\{1, \min_{z\in {{T}}^0}\gamma_z\}$. Clearly, the value $\gamma_{{S}}$ depends only on the metric of ${{S}}$. 
For a $g\times g$ matrix $\Pi$
denote $\|\Pi\|:=\sqrt{\sum_{1\le k,l\le g}|\Pi_{kl}|^2}$.

\begin{Convergence-Theorem-for-Period-Matrices}\label{th-main} 
For each $\delta>0$  
there are two constants $\mathrm{Const}_{\delta,{{S}}},\mathrm{const}_{\delta,{{S}}}>0$ depending only on $\delta$  and the metric of the surface ${{S}}$ such that for any triangulation ${{T}}$ of ${{S}}$ with the maximal edge length $h<\mathrm{const}_{\delta,{{S}}}$ 
and with the minimal face angle $>\delta$
we have
\begin{equation}\label{eq-error-estimate}
\|\Pi_{{{T}}}-\Pi_{{{S}}}\|
\le
\mathrm{Const}_{\delta,{{S}}}\cdot
\begin{cases}
h, &\text{if }\gamma_{{S}}>1/2;\\
h|\log h|, &\text{if } \gamma_{{S}}=1/2;\\
h^{2\gamma_{{S}}}, &\text{if }\gamma_{{S}}<1/2.
\end{cases}
\end{equation}
\end{Convergence-Theorem-for-Period-Matrices}

In particular, for a sequence of triangulations of the surface ${{S}}$ with the maximal edge length tending to zero 
and with face angles bounded from zero
the discrete period matrices converge to the period matrix $\Pi_{{S}}$. 
The approximation order in~\eqref{eq-error-estimate} 
agrees with numerical experiments; see 
Section~\ref{ssec-num}.
Theorem~\ref{th-main} is proved in Section~\ref{ssec:Convergence of period matrices}.

\subsection{Convergence of discrete Abelian integrals of the first kind}\label{ssec:Convergence of discrete Abelian integrals}

%
To state the next result we need the following notions.
We say that a discrete Abelian integral $\phi^l_{{{T}}}=
(\mathrm{Re}\,\phi^l_{{{T}}}\colon \widetilde{{{T}}}^0\to \mathbb{R},
\mathrm{Im}\,\phi^l_{{{T}}}\colon \widetilde{{{T}}}^2\to \mathbb{R})$
of the first kind is \emph{normalized} at a vertex $z\in \widetilde{{{T}}}^0$ and a face $w\in \widetilde{{{T}}}^2$, if 
$[\mathrm{Re}\,\phi^l_{{{T}}}](z)=[\mathrm{Im}\,\phi^l_{{{T}}}](w)=0$.
Similarly, 
we say that an Abelian integral $\phi^l_{{{S}}}\colon \widetilde{{{S}}}\to \mathbb{C}$ is \emph{normalized} at a point $z\in\widetilde{{{S}}}$, if
$\phi^l_{{{S}}}(z)=0$.

A triangulation ${{T}}$  is \emph{Delaunay}, if for each edge $e\in{{T}}^1$ we have $\alpha_e+\beta_e\le\pi$. 
Any polyhedral surface has an (essentially unique) Delaunay triangulation ~\cite{Bobenko-Springborn-07, A1-BOB-rivin_1994}.
Let $\{{{T}}_n\}$ be a sequence of triangulations of the surface ${{S}}$ such that \mscomm{all the faces are flat triangles}
(the piecewise-flat metric on the surface ${{S}}$ is fixed and does not depend on $n$). 
The sequence of triangulations $\{{{T}}_n\}$ is \emph{nondegenerate uniform}, if there is a constant $\mathrm{Const}$ (not depending on $n$) such that for each member of the sequence:
\begin{itemize}
\item[(A)] the angles of each face are greater than $1/\mathrm{Const}$;
\item[(D)] 
for each edge the sum of opposite angles in the two triangles containing the edge is less than $\pi-1/\mathrm{Const}$ (in particular, the triangulation is Delaunay);
%
\item[(U)] the number of vertices in an arbitrary intrinsic disk of radius equal to the maximal edge length is less than $\mathrm{Const}$.
\end{itemize}

A sequence of pairs of functions
$f_n=(\mathrm{Re}f_n\colon \widetilde{{{T}}}^0_n\to\mathbb{R},
\mathrm{Im}f_n\colon \widetilde{{{T}}}^2_n\to\mathbb{R})$ \emph{converges} to a function $f\colon\widetilde{{{S}}}\to\mathbb{C}$ \emph{uniformly on each compact subset}, if for each compact set $K\subset \widetilde{{{S}}}$ we have 
$$
\max_{z\in K\cap \widetilde{{{T}}}_n^0}|[\mathrm{Re}f_n](z)-\mathrm{Re} f(z)|\to 0 \quad\text{ and }\quad
\max_{xyz\in \widetilde{{{T}}}^2_n:K\cap xyz\ne\emptyset
}
|[\mathrm{Im}f_n](xyz)-\mathrm{Im}f(z)|\to 0 \quad\text{ as }n\to\infty.
$$

\begin{Convergence-Theorem-for-Abelian-integrals}
\label{th-Convergence of Abelian integrals}
Let $\{{{T}}_n\}$ be a nondegenerate uniform sequence of Delaunay triangulations of ${{S}}$ with maximal edge length approaching zero as $n\to\infty$. Let  $z_n\in\widetilde{{{T}}}_n^0$ be a sequence of vertices converging to a point $z_0\in\widetilde{{{S}}}$. Let $w_n\in\widetilde{{{T}}}_n^2$ be a sequence of faces 
\mscomm{with the vertices converging to~$z_0$.}
Then for each $1\le l\le g$ the discrete Abelian integrals $\phi^l_{{{T}}_n}=
(\mathrm{Re}\,\phi^l_{{{T}}_n}\colon \widetilde{{{T}}}_n^0\to \mathbb{R},
\mathrm{Im}\,\phi^l_{{{T}}_n}\colon \widetilde{{{T}}}_n^2\to \mathbb{R})$ of the 1st kind normalized at $z_n$ and $w_n$ converge uni\-formly on each compact set to the Abelian integral $\phi^l_{{{S}}}\colon \widetilde{{{S}}}\to \mathbb{C}$ of the 1st kind normalized at~$z_0$.
\end{Convergence-Theorem-for-Abelian-integrals}

This theorem is proved in Section~\ref{ssec:Convergence of discrete Abelian integrals of the first kind}. 



\subsection{Convergence of energy}

Our results are proved by energy estimates inspired by alternating-current networks. Let us state our main lemma establishing convergence of energy.

\emph{Multi-valued functions} \mscomm{$\widetilde{{{T}}}^0\to\mathbb{R}$}, $\widetilde{{{T}}}^2\to\mathbb{R}$, and $\widetilde{{{S}}}\to\mathbb{R}$ are defined analogously
\mscomm{to the multi-valued functions from Section~\ref{ssec:abelian}.} 
\mscomm{For each multi-valued function $u\colon\widetilde{{{T}}}^0\to\mathbb{R}$ and each edge $xy\in {{T}}^1$ the difference $u(x)-u(y)$ is well-defined.}
The \emph{energy} of the multi-valued function is
$$E_{{T}}(u):=\sum_{xy\in {{T}}^1}\mscomm{{c}}(xy)(u(x)-u(y))^2.$$
For each multi-valued function $u\colon \widetilde{{{S}}}\to\mathbb{R}$ and each point inside a face $w\in{{T}}^2$ the gradient $\nabla u$ is well-defined.
The \emph{energy} of the multi-valued function is
$$E_{{S}}(u):=\sum_{w\in{{T}}^2}\int_w |\nabla u|^2\,dxdy.$$ 

\begin{restriction-energy-convergence-lemma}\label{l-energy-convergence}
For each $\delta>0$ and each smooth multi-valued function $u\colon \widetilde{{{S}}}\to\mathbb{R}$ there are two constants $\mathrm{Const}_{u,\delta,{{S}}},\mathrm{const}_{u,\delta,{{S}}}>0$ 
 such that for any triangulation ${{T}}$ of ${{S}}$ with the maximal edge length $h<\mathrm{const}_{u,\delta,{{S}}}$ 
and with the minimal face angle $>\delta$
we have
\begin{equation*}
|E_{{{T}}}(u\left|_{\widetilde{{{T}}}^0}\right.)- E_{{S}}(u)|
\le
\mathrm{Const}_{u,\delta,{{S}}}\cdot
\begin{cases}
h, &\text{if }\gamma_{{S}}>1/2;\\
h|\log h|, &\text{if } \gamma_{{S}}=1/2;\\
h^{2\gamma_{{S}}}, &\text{if }\gamma_{{S}}<1/2.
\end{cases}
\end{equation*}
\end{restriction-energy-convergence-lemma}

Such approximation order is informally explained as follows. There are two competing main sources of energy approximation error. In the flat regions apart from singular points there is a nice approximation with the error of order $h$. In the $h$-neighborhood of conical singularities there is essentially no approximation, and the error has order of the energy itself, i.e., $h^{2\gamma_{{S}}}$.  Depending on whether $\gamma_{{S}}>1/2$ or $\gamma_{{S}}<1/2$, one of the error sources dominates.

The assumption on the minimal face angle in the lemma cannot be dropped; see Example~\ref{ex-counterexample}. 
The lemma is proved in Section~\ref{ssec-energy-convergence}.

\section{Discrete period matrices}\label{sec-period}

In this section we prove basic properties of discrete period matrices. 

\subsection{Riemann Bilinear Identity}

We start with the following 
result. In what follows fix an arbitrary orientation of each edge 
of~${{{T}}}^1$.

\begin{Riemann-bilinear-identity} \label{th-Riemann2}
Let $u\colon \widetilde{{{T}}}^0\to\mathbb{R}$ and $u'\colon \widetilde{{{T}}}^2\to\mathbb{R}$ be two multi-valued functions with periods
$A_1,\dots,A_{g},B_1,\dots,B_{g}$ and $A'_1,\dots,A'_{g},B'_1,\dots,B'_{g}$, respectively. Then
$$
\sum_{e\in {{{T}}^1}}(u'(l_e)-u'(r_e))(u(h_e)-u(t_e))
= \sum_{k=1}^g (A_k B_{k}'-B_k A_k').
$$
\end{Riemann-bilinear-identity}

The left-hand side of this identity makes sense for an arbitrary cell decomposition~${{T}}$, not necessarily a triangulation and not necessarily equipped with a metric. Such generalization appears in the following example.

\begin{figure}[htb]
\definecolor{ffqqqq}{rgb}{1,0,0}
\definecolor{qqttcc}{rgb}{0,0.2,0.8}
\definecolor{uuuuuu}{rgb}{0.27,0.27,0.27}
\begin{tikzpicture}[line cap=round,line join=round,>=triangle 45,x=0.5cm,y=0.5cm]
\clip(-0.4,-2.5) rectangle (17.25,11.35);
\draw [->] (9.83,2.83) -- (9.83,6.83);
\draw [->] (9.83,6.83) -- (7,9.66);
\draw [->] (3,9.66) -- (0.17,6.83);
\draw [->] (3,0) -- (7,0);
\draw (0.73,2.62) node[anchor=north west] {$\ddots$};
\draw [->] (7,0) -- (9.83,2.83);
\draw [->] (0.17,6.83) -- (0.17,2.83);
\draw (4.38,10.09) node[anchor=north west] {$\ldots$};
\draw (7,0)-- (8,-1);
\draw (9.83,2.83)-- (10.86,1.86);
\draw (9.83,2.83)-- (10.92,2.82);
\draw (9.83,6.83)-- (10.92,6.8);
\draw (8.52,1.23) node[anchor=north west] {$r_{e_{4k-2}}$};
\draw (10,5.42) node[anchor=north west] {$r_{e_{4k-1}}$};
\draw (1.68,5.91) node[anchor=north west] {$l_{e_{4k-2}}=l_{e_{4k-1}}$};
\draw [shift={(10.61,4.92)},dash pattern=on 6pt off 6pt,color=qqttcc]  plot[domain=-1.72:1.75,variable=\t]({1*1.64*cos(\t r)+0*1.64*sin(\t r)},{0*1.64*cos(\t r)+1*1.64*sin(\t r)});
\draw [shift={(8.21,-0.82)},dash pattern=on 6pt off 6pt,color=ffqqqq]  plot[domain=1.19:1.94,variable=\t]({1*6.92*cos(\t r)+0*6.92*sin(\t r)},{0*6.92*cos(\t r)+1*6.92*sin(\t r)});
\draw [shift={(10.08,4.66)},dash pattern=on 6pt off 6pt,color=qqttcc]  plot[domain=3.21:4.31,variable=\t]({1*4.47*cos(\t r)+0*4.47*sin(\t r)},{0*4.47*cos(\t r)+1*4.47*sin(\t r)});
\draw [shift={(8.97,0.88)},dash pattern=on 6pt off 6pt,color=ffqqqq]  plot[domain=-2.54:0.95,variable=\t]({1*1.76*cos(\t r)+0*1.76*sin(\t r)},{0*1.76*cos(\t r)+1*1.76*sin(\t r)});
\draw [->] (8.04,0.68) -- (8.34,0.54);
\draw [->] (10.48,5.72) -- (10.8,5.6);
\draw [->] (10.86,3.3) -- (10.36,3.3);
\draw [->] (10.41,1.91) -- (10,2.31);
\fill [color=black] (3,0) circle (1.5pt);
\fill [color=black] (7,0) circle (1.5pt);
\draw[color=black] (6.36,-0.84) node {$t_{e_{4k-2}}$};
\fill [color=uuuuuu] (9.83,2.83) circle (1.5pt);
\draw[color=uuuuuu] (13.51,2.22) node {$h_{e_{4k-2}}=t_{e_{4k-1}}$};
\fill [color=uuuuuu] (9.83,6.83) circle (1.5pt);
\draw[color=uuuuuu] (10.86,7.62) node {$h_{e_{4k-1}}$};
\fill [color=uuuuuu] (7,9.66) circle (1.5pt);
\fill [color=uuuuuu] (3,9.66) circle (1.5pt);
\fill [color=uuuuuu] (0.17,6.83) circle (1.5pt);
\fill [color=uuuuuu] (0.17,2.83) circle (1.5pt);
\draw[color=black] (8.83,4.6) node {$e_{4k-1}$};
\draw[color=black] (8.02,7.71) node {$e_{4k}$};
\draw[color=black] (5.14,0.42) node {$e_{4k-3}$};
\draw[color=black] (7.8,1.95) node {$e_{4k-2}$};
\draw[color=qqttcc] (13.51,5.19) node {$d_{\widehat{\alpha}_k}$};
\draw[color=ffqqqq] (8.79,6.63) node {$d_{\widehat{\beta}_k}$};
\draw[color=qqttcc] (6.85,2.62) node {$d_{\widehat{\alpha}_k}$};
\draw[color=ffqqqq] (11.67,-0.43) node {$d_{\widehat{\beta}_k}$};
\end{tikzpicture}
\caption{Action of deck transformations on the universal covering of the canonical cell decomposition; see Example~\ref{ex-canonic} for the details.}
\label{fig:canonic}
\end{figure}
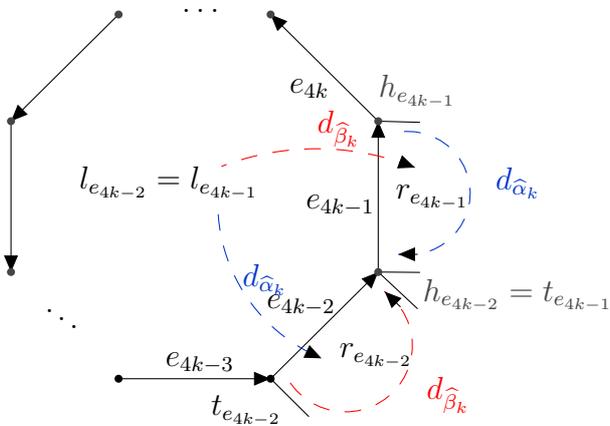

\begin{example}\label{ex-canonic} 
\mscomm{A} \emph{canonical cell decomposition} ${{T}}$ of a surface ${{S}}$ of genus $g$ is obtained from a $4g$-gon with the sides $e_1\dots e_{4g}$ oriented counterclockwise by gluing together the pairs of sides $e_{4k-3}, e_{4k-1}$ and $e_{4k-2}, e_{4k}$ for each $k=1,\dots,g$
with orientation reversal.
The $4g$-gon itself is considered as a face of the universal cover $\widetilde{{{T}}}$; see Figure~\ref{fig:canonic}. 
The paths $\alpha_k:=p e_{4k-3}$ and $\beta_k:=p e_{4k-2}$ for $k=1,\dots,g$ \emph{going along the sides of the $4g$-gon} form a standard basis of the fundamental group of ${{S}}$. 
It is easy to see that there are deck transformations $d_{\widehat{\alpha}_k},d_{\widehat{\beta}_k}\colon \widetilde{{{T}}}\to \widetilde{{{T}}}$ such that
\begin{align*}
d_{\widehat{\alpha}_k}h_{e_{4k-1}}&=t_{e_{4k-1}}; &     
d_{\widehat{\alpha}_k}l_{e_{4k-2}}&=r_{e_{4k-2}};\\
d_{\widehat{\beta}_k}t_{e_{4k-2}}&=h_{e_{4k-2}}; & 
d_{\widehat{\beta}_k}l_{e_{4k-1}}&=r_{e_{4k-1}};
\end{align*}
where $\widehat{\alpha}_k$ and $\widehat{\beta}_k$ are appropriate elements of $\pi_1({{S}},pz_0)$ conjugate to ${\alpha}_k$ and ${\beta}_k$, respectively.
Indeed, if the base point $z_0=h_{e_{4k-1}}$ then we can just take $\widehat{\alpha}_k=\alpha_k$, and change of the base point leads to conjugation in $\pi_1({{S}},pz_0)$.
This implies that 
\begin{align*}
u(t_{e_{4k-1}})-u(h_{e_{4k-1}})&=A_k; &     
u'(r_{e_{4k-2}})-u'(l_{e_{4k-2}})&=A'_k;\\
u(h_{e_{4k-2}})-u(t_{e_{4k-2}})&=B_k; &
u'(r_{e_{4k-1}})-u'(l_{e_{4k-1}})&=B'_k.
\end{align*}
Thus for the canonical cell decomposition Riemann Bilinear Identity~\ref{th-Riemann2}
is checked directly:
\begin{align*}
\sum_{e\in {{{T}}^1}}(u'(l_e)-u'(r_e))(u(h_e)-u(t_e))
&=\sum_{k=1}^{g} (u'(l_{e_{4k-1}})-u'(r_{e_{4k-1}}))(u(h_{e_{4k-1}})-u(t_{e_{4k-1}}))\\
&+\sum_{k=1}^{g} 
(u'(l_{e_{4k-2}})-u'(r_{e_{4k-2}}))(u(h_{e_{4k-2}})-u(t_{e_{4k-2}}))\\
&=\sum_{k=1}^g (A_k B_{k}'-B_k A_k').
\end{align*}
\end{example}

Now we proceed to general triangulations.
Note that there are triangulations which are not subdivisions of the canonical cell decomposition, e.g., a triangulation of a surface of   genus $g>3$ with all vertex degrees $<2g$. To prove the identity for general triangulations we need the following auxiliary assertion.



\begin{stokes-formula}\label{l-Stokes}
Let $\widehat{{{T}}}\subset \widetilde{{{T}}}$ be a triangulated polygon. Then
\begin{equation*}
\sum_{e\in\widehat{{{T}}}^1}
(u'(l_e)-u'(r_e))(u(h_e)-u(t_e))
=
\sum_{e\in(\partial\widehat{{{T}}})^1}
u'(r_e)(u(t_e)-u(h_e)),
\end{equation*}
where each edge $e\in (\partial\widehat{{{T}}})^1$ is oriented so that $l_e\subset\widehat{{{T}}}$.
\end{stokes-formula}

\begin{proof} In the case when the polygon $\widehat{{{T}}}$ is a single face of the triangulation $\widetilde{{{T}}}$ Stokes' formula is checked directly.
Sum such Stokes' formulas over all the faces of the triangulated polygon $\widehat{{{T}}}$.
Each interior edge $e$ appears two times with each of the two possible orientations. Hence it contributes
$2(v(l_e)-v(r_e))(u(h_e)-u(t_e))$ to the left-hand side of the sum, and
$(v(r_e)-v(l_e))(u(t_e)-u(h_e))$ to the right-hand side. Canceling the repeating terms, we get the required formula.
\end{proof}

\begin{proof}[Proof of Riemann Bilinear Identity~\ref{th-Riemann2}]
We are going to transform the triangulation ${{T}}$ without changing the left-hand side of the required identity, and this way reduce the 
identity to Example~\ref{ex-canonic}. An \emph{edge subdivision} is replacement of a pair of adjacent faces $xyz$ and $xyw$ of the triangulation by $4$ new faces $xzv$, $zvy$, $yvw$, $wvx$ sharing a common new vertex $v$. It is well-known that two arbitrary triangulations can be transformed to each other by edge subdivisions and their inversions.

Let us check that  the left-hand side of the required identity is invariant under an edge subdivision 
and arbitrary extension of the functions $u$ and $u'$ to the obtained new vertices and faces.  Indeed, the contributions of all the edges  not contained in the quadrilateral $xzyw$ to the left-hand side are not affected by the edge subdivision. Stokes' Formula~\ref{l-Stokes} immediately implies that the total contributions of the edges  contained in the quadrilateral $xzyw$ are the same before and after the edge subdivision. So the left-hand side is invariant.

Thus without loss of generality we may assume that the triangulation ${{T}}$ is a subdivision of a canonical cell decomposition with the given paths $\alpha_1,\dots,\alpha_g,\beta_1,\dots,\beta_g$ following the sides of the $4g$-gon.
By Stokes' Formula~\ref{l-Stokes} it follows that the left-hand side of required identity is invariant under adjustment of the value of the function $u'$ at one particular face and its images under all possible deck transformations. 
Thus without loss of generality we may assume that $u'$ assigns the same values at all the faces inside the $4g$-gon of Example~\ref{ex-canonic}. Then neither of the edges lying inside the $4g$-gon contributes to the left-hand side of the required identity. This means that the left-hand side equals to that of the canonical cell decomposition. By Example~\ref{ex-canonic} the identity follows.
\end{proof}

\subsection{Energy principles}

We are going to use the following basic properties of energy; cf.~\cite[Principles~2.1, 2.2, 5.6]{A1-Skopenkov}.

\begin{energy-conservation-principle}
\label{th-energy-conservation}
Let $f=(\mathrm{Re}f\colon\widetilde{{{T}}}^0\to\mathbb{R},\mathrm{Im}f\colon\widetilde{{{T}}}^2\to\mathbb{R})$ be a discrete Abelian integral of the 1st kind with periods $A_1,\dots,A_{g},B_1,\dots,B_{g}$. Then $$E_{{T}}(\mathrm{Re}f)=-\Imaginary\sum_{k=1}^g A_k\bar B_{k}.$$
\end{energy-conservation-principle}

\begin{proof} By equation~\eqref{eq-def-analytic2} and the Riemann Bilinear Identity~\ref{th-Riemann2} we have
\begin{align*}
E_{{{T}}}(\mathrm{Re}f)
&=\sum_{e\in {{T}}^1}\mscomm{{c}}(e)(\mathrm{Re}f(h_e)-\mathrm{Re}f(t_e))^2\\
&=\sum_{e\in {{T}}^1}
(\mathrm{Im}f(l_e)-\mathrm{Im}f(r_e))(\mathrm{Re}f(h_e)-\mathrm{Re}f(t_e))\\
&=\sum_{k=1}^g(\mathrm{Re}A_k \mathrm{Im}B_{k} - \mathrm{Re}B_k \mathrm{Im}A_{k})\\
&=-\Imaginary\sum_{k=1}^g A_k\bar B_{k},\\[-1.6cm]
\end{align*}
\end{proof}

\begin{energy-positivity-principle}\label{convexity-principle}
For each nonconstant multi-valued function $u\colon \widetilde{{{T}}}^0\to\mathbb{R}$ we have $E_{{T}}(u)>0$.
\end{energy-positivity-principle}

\begin{proof}  
For $u\ne \mathrm{const}$ using the formula $\cot zxy=\frac{1-\cot xyz\cot yzx}{\cot xyz + \cot yzx}$ we get
\begin{align*}
E_{{{T}}}(u)
&=\sum_{xy\in {{T}}^1}\mscomm{{c}}(xy)(u(x)-u(y))^2\\
&=\frac{1}{2}\sum_{xyz\in {{T}}^2}
\left( \cot xyz (u(x)-u(z))^2 + \cot yzx (u(y)-u(x))^2 + 
\cot zxy (u(z)-u(y))^2\right)\\
&=\frac{1}{2}\sum_{xyz\in {{T}}^2}
\frac{\left(\cot xyz (u(x)-u(z)) + \cot yzx (u(x)-u(y))\right)^2 + (u(z)-u(y))^2}
{\cot xyz + \cot yzx}>0.
\end{align*}
(Another proof is by application of Interpolation Lemma~\ref{l-interpolation-1} below).
\end{proof}




\begin{variational-principle}\label{l-variational-principle5} 
A multi-valued discrete harmonic function has minimal energy
among all multi-valued functions with the same periods.
\end{variational-principle}


\begin{proof}
Consider the finite-dimensional affine space of multi-valued functions $\widetilde{{{T}}}^0\to\mathbb{R}$ with the same periods as a given multi-valued discrete harmonic function $u\colon\widetilde{{{T}}}^0\to\mathbb{R}$.
By Energy Positivity Principle~\ref{convexity-principle} it follows that the energy $E_{{T}}(\cdot)$ is a (nonstrictly) convex functional on this space. Since $u\colon\widetilde{{{T}}}^0\to\mathbb{R}$ 
is discrete harmonic it follows \mscomm{by differentiation} that it is a critical point of the functional $E_{{T}}(\cdot)$. Thus $u$ is a \mscomm{global} minimum.
\end{proof}

If $\mscomm{{c}}(xy)\ne0$ for each edge $xy\in {{T}}^1$ 
then the \emph{energy} of a multi-valued function $v\colon \widetilde{{{T}}}^2\to\mathbb{R}$ is defined by the formula
$E_{{T}}(v):=\sum_{xy\in {{T}}^1}(v(x)-v(y))^2/\mscomm{{c}}(xy)$. 
Identity~\eqref{eq-def-analytic2} immediately implies the following \mscomm{assertion}.

\begin{conjugate-functions-principle}
\label{th-energy-conjugation}
Let $f=(\mathrm{Re}f\colon\widetilde{{{T}}}^0\to\mathbb{R},\mathrm{Im}f\colon\widetilde{{{T}}}^2\to\mathbb{R})$ be a discrete Abelian integral of the 1st kind. Then $E_{{T}}(\mathrm{Re}f)=E_{{T}}(\mathrm{Im}f).$
\end{conjugate-functions-principle}


\subsection{Existence and Uniqueness Theorem}\label{ssec-exist-unique}

\begin{corollary}\label{cor-zero-periods2}
Let $f=(\mathrm{Re}f\colon\widetilde{{{T}}}^0\to\mathbb{R},\mathrm{Im}f\colon\widetilde{{{T}}}^2\to\mathbb{R})$ be a discrete Abelian integral of the 1st kind with the periods $A_1,\dots,A_{g},B_1,\dots,B_{g}$. Then the following $3$ conditions are equivalent:
\begin{enumerate}
 \item $f=\mathrm{const}$;
 \item $A_1=\dots=A_{g}=0$;
 \item $\mathrm{Re}A_1=\dots=\mathrm{Re}A_{g}=
     \mathrm{Re}B_1=\dots=\mathrm{Re}B_{g}=0$.
\end{enumerate}
\end{corollary}

\begin{proof} The implications 1$\Rightarrow$2 and 1$\Rightarrow$3 are obvious. Let us prove that 2$\Rightarrow$1; the implication 3$\Rightarrow$1 is proved analogously.
By Energy Conservation Principle~\ref{th-energy-conservation} we have
$$E_{{{T}}}(\mathrm{Re}f)=\sum_{k=1}^g(\mathrm{Re}A_k\mathrm{Im} B_{k}-\mathrm{Re}B_{k}\mathrm{Im} A_k)=0.$$
Hence by Energy Positivity Principle~\ref{convexity-principle} we get $\mathrm{Re}f=\mathrm{const}$.
By~\eqref{eq-def-analytic2} 
we get also $\mathrm{Im}f=\mathrm{const}$. Hence $f=\mathrm{const}$.
\end{proof}

\begin{proof}[Proof of First Existence and Uniqueness Theorem~\ref{th-fund}] \emph{Uniqueness}. Let $f,f'$ be two discrete Abelian integrals of the 1st kind as required. Then $f-f'$ is a discrete Abelian integral of the 1st kind with vanishing A-periods. By Corollary~\ref{cor-zero-periods2}(2$\Rightarrow$1) it follows that $f-f'=\mathrm{const}$. 

\emph{Existence.} Denote by $f=(\mathrm{Re}f\colon\widetilde{{{T}}}^0\to\mathbb{R},\mathrm{Im}f\colon\widetilde{{{T}}}^2\to\mathbb{R})$ an unknown multi-valued function. The real and imaginary parts of the A-periods $A_1,\dots,A_{g}$ are considered as parameters.

Consider the following system of linear equations.
\mscomm{Identify ${{T}}^0,{{T}}^1,
{{T}}^2$ with some subsets of $\widetilde{{{T}}}^0,\widetilde{{{T}}}^1,
\widetilde{{{T}}}^2$ which project bijectively to ${{T}}^0,{{T}}^1,
{{T}}^2$, respectively.}  Let the variables be the values of the function $f=(\mathrm{Re}f\colon\widetilde{{{T}}}^0\to\mathbb{R},\mathrm{Im}f\colon\widetilde{{{T}}}^2\to\mathbb{R})$ at the elements of the sets ${{{T}}}^0\subset \widetilde{{{T}}}^0$, ${{{T}}}^2\subset \widetilde{{{T}}}^2$,
and also unknown real and imaginary parts of the B-periods $B_{1}$, \dots, $B_{g}$ of this function. Express the values of the function
$\mathrm{Re}f\colon\widetilde{{{T}}}^0\to\mathbb{R}$ 
at the
vertices 
not belonging to ${{{T}}}^0\subset \widetilde{{{T}}}^0$ through the values at ${{{T}}}^0$ and the periods $A_1,\dots,A_{g}$, $B_{1}, \dots, B_{g}$. Similarly, express the values of the function
$\mathrm{Im}f\colon\widetilde{{{T}}}^2\to\mathbb{R}$
at the
faces 
not belonging to ${{{T}}}^2\subset \widetilde{{{T}}}^2$ through the values at ${{{T}}}^2$ and the periods $A_1,\dots,A_{g}$, $B_{1}$, \dots, $B_{g}$. For each edge $e\in {{T}}^1$ choose an orientation and write one linear equation~\eqref{eq-def-analytic2}. Also pick up a vertex $z\in \widetilde{{{T}}}^0$, a face $w\in\widetilde{{{T}}}^2$, and write two equations $\mathrm{Re}f(z)=0$, $\mathrm{Im}f(w)=0$.

 We have written a system of $|{{T}}^1|+2$ linear equations in $|{{T}}^0|+|{{T}}^2|+2g$ variables, where $|F|$ denotes the number of elements in a finite set $F$. By the Euler formula the number of equations in the system equals the number of variables.
The free terms of these equations are linear combinations of the parameters $\mathrm{Re}A_1,\mathrm{Im}A_1,\dots,\mathrm{Re}A_g,\mathrm{Im}A_{g}$.
By the uniqueness part of the theorem it follows that the homogeneous system obtained for $A_1=\dots=A_{g}=0$ has only the trivial solution. Then by the finite-dimensional Fredholm alternative it follows that for arbitrary $A_1,\dots,A_{g}\in\mathbb{C}$ the inhomogeneous system has at least one solution. Thus there exists a discrete Abelian integral of the 1st kind as required.
\end{proof}

The following theorem is proved analogously using Corollary~\ref{cor-zero-periods2}(3$\Rightarrow$1).

\begin{Second-Existence-and-Uniqueness-Theorem} \label{th-fund3}
For each $P=(A_1,\dots,A_{g},B_1,\dots,B_{g})\in\mathbb{R}^{2g}$
there 
is a unique (up to constant) discrete Abelian integral $\phi_{{{T}},P}=(\mathrm{Re}\,\phi_{{{T}},P}\colon\widetilde{{{T}}}^0\to\mathbb{R},\mathrm{Im}\,\phi_{{{T}},P}\colon\widetilde{{{T}}}^2\to\mathbb{R})$ of the 1st kind 
whose periods have \emph{real parts}  $A_1,\dots,A_{g},B_1,\dots,B_{g}$, respectively.
\end{Second-Existence-and-Uniqueness-Theorem}

\subsection{Discrete period matrices}

For each $l=1,\dots,{g}$ denote by $\phi^l_{{{T}}^*}=(\mathrm{Re}\,\phi^l_{{{T}}^*}\colon\widetilde{{{T}}}^0\to\mathbb{R},\mathrm{Im}\,\phi^l_{{{T}}^*}\colon\widetilde{{{T}}}^2\to\mathbb{R})$ the unique (up to constant) Abelian integral of the 1st kind with A-periods given by the formula $A_k=i\delta_{kl}$, where $k=1,\dots,g$. 
The $g\times g$ matrix $\Pi_{{{T}}^*}$ whose $l$-th column is the B-periods of $\phi^l_{{{T}}^*}$ divided by $i$ is called the \emph{dual period matrix} of the triangulation~${{T}}$.

\begin{example} The map taking the vector of A-periods of a discrete Abelian integral of the first kind to the vector of its B-periods is not necessarily $\mathbb{C}$-linear, and thus $\Pi_{{{T}}^*}\ne\Pi_{{{T}}}$ in general. For instance, if ${{T}}$ is obtained from the side surface of a regular square pyramid by identifying the opposite sides of the base with orientation reversal then
$\Pi_{{{T}}}=2i/\sqrt{3}$ whereas $\Pi_{{{T}}^*}=i\sqrt{3}/2$.
\end{example}

\begin{period-matrix-lemma}\label{l-properties-of-period-matrix2}
The matrices $\mathrm{Im}\,\Pi_{{T}}$ and $\mathrm{Im}\,\Pi_{{{T}}^*}$ are symmetric positively definite, and
$\mathrm{Re}\,\Pi_{{{T}}^*}=(\mathrm{Re}\,\Pi_{{T}})^\mathrm{T}$.
\end{period-matrix-lemma}

\begin{remark} The \emph{discrete period matrix} by Mercat \cite{Mercat-02} is $\Pi_{{{Q}}}:=(\Pi_{{{T}}}+\Pi_{{{T}}^*})/2$. The matrix $\Pi_{{{Q}}}$ is symmetric with positively definite imaginary part. Our proof of Convergence Theorem for Period Matrices~\ref{th-main} implies an analogue of error estimate~\eqref{eq-error-estimate} for this matrix as well:
\begin{equation*}
\|\Pi_{{Q}}-\Pi_{{S}}\|\le \mathrm{Const}_{\delta,{{S}}}\cdot
\begin{cases}
h, &\text{if }\gamma_{{S}}>1/2;\\
h|\log h|, &\text{if } \gamma_{{S}}=1/2;\\
h^{2\gamma_{{S}}}, &\text{if }\gamma_{{S}}<1/2.
\end{cases}
\end{equation*}
\end{remark}

For the proof of the lemma we need the following auxiliary assertion.

\begin{lemma}\label{cl-symmetry} Let 
$
f=(\mathrm{Re}f\colon \widetilde{{{T}}}^0\to\mathbb{R},
\mathrm{Im}f\colon \widetilde{{{T}}}^2\to\mathbb{R})
$
and
$
f'=(\mathrm{Re}f'\colon \widetilde{{{T}}}^0\to\mathbb{R},
\mathrm{Im}f'\colon \widetilde{{{T}}}^2\to\mathbb{R})
$ 
be two discrete Abelian integrals of the 1st kind with the periods
$A_1,\dots,A_{g},B_1,\dots,B_{g}
$
and
$A'_1,\dots,A'_{g}$, $B'_1$, \dots, $B'_{g}$,
respectively. Then
$$
\mathrm{Im}\,\sum_{k=1}^g (A_k B_{k}'-B_k A_k')=0.
$$
\end{lemma}  

\begin{proof} 
Using Riemann Bilinear Identity~\ref{th-Riemann2} and then formula~\eqref{eq-def-analytic2} to cancel repeating terms we get
\begin{align*}
\mathrm{Im}\,\sum_{k=1}^g (A_k B_{k}'-B_k A_k')&=
\sum_{k=1}^g
(\mathrm{Re}A_k\mathrm{Im}B_k'-
\mathrm{Re}B_k\mathrm{Im}A_k'-
\mathrm{Re}A_k'\mathrm{Im}B_k+
\mathrm{Re}B_k'\mathrm{Im}A_k)\\
&=\sum_{e\in{{T}}^1}(\mathrm{Im}f'(l_e)-\mathrm{Im}f'(r_e))
(\mathrm{Re}f(h_e)-\mathrm{Re}f(t_e)) \\
&- \sum_{e\in{{T}}^1}
(\mathrm{Im}f(l_e)-\mathrm{Im}f(r_e))
(\mathrm{Re}f'(h_e)-\mathrm{Re}f'(t_e))\\
&=0.\\[-1.6cm]
\end{align*}
\end{proof}

\begin{proof}[Proof of Period Matrix Lemma~\ref{l-properties-of-period-matrix2}]
Applying Lemma~\ref{cl-symmetry} for $f:=\phi^i_{{{T}}}$, $f':=\phi^j_{{{T}}}$ we get
$
(\mathrm{Im}\Pi_{{{T}}})_{ij}-
(\mathrm{Im}\Pi_{{{T}}})_{ji}=0,
$
that is, $\mathrm{Im}\,\Pi_{{T}}$ is symmetric.
Applying Lemma~\ref{cl-symmetry} for $f:=\phi^i_{{{T}}}$, $f':=\phi^j_{{{T}}^*}$ we get
$
(\mathrm{Re}\Pi_{{{T}}*})_{ij}-
(\mathrm{Re}\Pi_{{{T}}})_{ji}=0
$
that is, $\mathrm{Re}\,\Pi_{{{T}}^*}=(\mathrm{Re}\,\Pi_{{T}})^\mathrm{T}$.

Finally, let $f$ be a discrete Abelian integral of the 1st kind whose A-periods 
$A_1$, \dots, $A_g$ are arbitrary \emph{real} numbers not vanishing simultaneously. Then the B-periods of this integral are equal to $B_l=\sum_{k=1}^g  (\Pi_{{{T}}})_{kl} A_k$. Hence
by Energy Conservation Principle~\ref{th-energy-conservation}
and Energy Positivity Principle~\ref{convexity-principle} it follows that
$$
\sum_{1\le k,l\le g} (\mathrm{Im}\Pi_{{{T}}})_{kl} A_k A_l=
-\sum_{k=1}^g\Imaginary (A_k\bar B_{k})=
E_{{{T}}}(\mathrm{Re}f)>0.
$$
Thus $\mathrm{Im}\Pi_{{{T}}}$ is positively definite. Analogously, $\mathrm{Im}\Pi_{{{T}}^*}$ is symmetric and positively definite.
\end{proof}

Denote $u_{{{T}},P}:=\mathrm{Re}\,\phi_{{{T}},P}$, where $\phi_{{{T}},P}$ is the discrete Abelian integral of the 1st kind defined in Second Existence and Uniqueness Theorem~\ref{th-fund3} for each vector  $P\in\mathbb{R}^{2g}$.

\begin{energy-matrix-lemma}\label{l-relation2} 
The energy $E_{{T}}(u_{{{T}},P})$ is a quadratic form in the vector $P\in\mathbb{R}^{2g}$ with the block matrix
\begin{equation*}
E_{{{T}}}:=
\left(\begin{matrix}
\Real \Pi_{{{T}}^*}(\Imaginary \Pi_{{{T}}^*})^{-1}\Real\Pi_{{T}}+\Imaginary\Pi_{{T}} & -(\Imaginary \Pi_{{{T}}^*})^{-1}\Real\Pi_{{{T}}} \\
-\Real\Pi_{{{T}}^*}(\Imaginary \Pi_{{{T}}^*})^{-1} & (\Imaginary \Pi_{{{T}}^*})^{-1}
\end{matrix}\right).
\end{equation*}
\end{energy-matrix-lemma}

\begin{proof} Denote by 
$A_1,\dots,A_g,B_1\dots,B_g\in\mathbb{C}$ the periods of the discrete Abelian integral $\phi_{{{T}},P}$ of the 1st kind. 
Denote $A:=(A_1,\dots,A_g)$, $B:=(B_1,\dots,B_g)$.
Then $P=(\mathrm{Re}A, \mathrm{Re}B)=(\mathrm{Re}A_1,\dots,\mathrm{Re}A_g,
\mathrm{Re}B_1\dots,\mathrm{Re}B_g)$.
By First Existence and Uniqueness Theorem~\ref{th-fund} we get 
$$\phi_{{{T}},P}=\mathrm{Re}A_1\,\phi^1_{{{T}}}+\dots+\mathrm{Re}A_g\,\phi^g_{{{T}}}+\mathrm{Im}A_1\,\phi^1_{{{T}}^*}+\dots+\mathrm{Im}A_g\,\phi^g_{{{T}}^*}
+\mathrm{const}.$$ Thus
\begin{align*}
\mathrm{Re}B&=\mathrm{Re}\Pi_{{{T}}}\mathrm{Re}A-
\mathrm{Im}\Pi_{{{T}}^*}\mathrm{Im}A;\\
\mathrm{Im}B&=\mathrm{Im}\Pi_{{{T}}}\mathrm{Re}A+
\mathrm{Re}\Pi_{{{T}}^*}\mathrm{Im}A.
\end{align*}
The matrix $\mathrm{Im}\Pi_{{{T}}^*}$ is invertible by Period Matrix Lemma~\ref{l-properties-of-period-matrix2}. Solving in $\mathrm{Im}A$ and $\mathrm{Im}B$ we get
\begin{align*}    
\mathrm{Im}A &= (\mathrm{Im}\Pi_{{{T}}^*})^{-1}\,
\mathrm{Re}\Pi_{{{T}}}\,\mathrm{Re}A-
(\mathrm{Im}\Pi_{{{T}}^*})^{-1}\mathrm{Re}B;\\
\mathrm{Im}B &= \left(
\mathrm{Im}\Pi_{{{T}}}+\mathrm{Re}\Pi_{{{T}}^*}(\mathrm{Im}\Pi_{{{T}}^*})^{-1}\,
\mathrm{Re}\Pi_{{{T}}}
\right)\mathrm{Re}A-
\mathrm{Re}\Pi_{{{T}}^*}(\mathrm{Im}\Pi_{{{T}}^*})^{-1}\mathrm{Re}B.
\end{align*}
Applying Energy Conservation Principle~\ref{th-energy-conservation} for $f:=\phi_{{{T}},P}$ we get the required expression for the matrix of $E_{{T}}(u_{{{T}},P})$.
\end{proof}




Let us conclude the section by a continuous counterpart of Energy Matrix Lemma~\ref{l-relation2}, which is proved analogously.
For $P=(A_1,\dots,A_{g},B_1,\dots,B_{g})\in\mathbb{R}^{2g}$ let $\phi_{{{S}},P}\colon \widetilde{{{S}}}\to \mathbb{C}$ be an Abelian integral of the first kind whose periods have \emph{real parts} 
$A_1,\dots,A_{g},B_1,\dots,B_{g}$. Denote $u_{{{S}},P}:=\mathrm{Re}\,\phi_{{{S}},P}$. 

\begin{lemma}\label{l-relation3} The energy $E_{{S}}(u_{{{S}},P})$ is a quadratic form in the vector $P\in\mathbb{R}^{2g}$ with the block matrix
\begin{equation*}
E_{{{S}}}:=
\left(\begin{matrix}
\Real \Pi_{{{S}}}(\Imaginary \Pi_{{{S}}})^{-1}\Real\Pi_{{S}}+\Imaginary\Pi_{{S}} & -(\Imaginary \Pi_{{{S}}})^{-1}\Real\Pi_{{{S}}} \\
-\Real\Pi_{{{S}}}(\Imaginary \Pi_{{{S}}})^{-1} & (\Imaginary \Pi_{{{S}}})^{-1}
\end{matrix}\right).
\end{equation*}
\end{lemma}

The last two lemmas show that convergence of discrete period matrices reduces to convergence of the energy form. We are going to establish the latter in the next section. 

\section{Convergence of energy and discrete period matrices}

In this section we prove convergence of discrete energy and discrete period matrices to their continuous counterparts. 
In what follows we use the following notation:
\begin{itemize}
\item $\mathrm{Const}$ is a positive constant (like $2\pi$ or $10^{100}$) that does not depend on any parameters of a configuration under consideration (e.g., $h$, $\gamma_{{S}}$, the shape of ${{T}}$);
$\mathrm{Const}$ may denote distinct constants at distinct places of the text (e.g., in the sides of a formula like $2\cdot\mathrm{Const}\le\mathrm{Const}$);
\item $\mathrm{Const}_{x,y,z}$ is a positive constant depending only on the parameters $x,y,z$;
\item 
${\|D^k u(z)\|}:=\max_{0\le j\le k} \left|\frac{\partial^k u}{\partial^j x\,\partial^{k-j}y}(z)\right|$.
\end{itemize}

\subsection{Convergence of energy in a triangle}\label{ssec-energy-convergence-plane}

We start with several well-known lemmas; cf.~\cite[Lemmas~4.2, 4.5, and~2.3]{A1-Skopenkov}.
Throughout this subsection $\Delta$ is a triangle in the plane,
$x,y,z$ are its vertices, and $u\colon\Delta\to\mathbb{R}$ is a smooth function which smoothly extends to a neighborhood of $\Delta$.
The \emph{energy} and the \emph{discrete energy} of the function $u\colon \Delta\to\mathbb{R}$ are $E_{{{S}}_\Delta}(u):=\int_{\Delta}|\nabla u|^2 dxdy$ and 
$$E_{{{T}}_\Delta}(u):=
\frac{1}{2}\cot xyz \cdot (u(z)-u(x))^2+\frac{1}{2}\cot yzx \cdot (u(x)-u(y))^2+\frac{1}{2}\cot zxy \cdot (u(y)-u(z))^2,
$$
respectively. The \emph{interpolation} of the function $u\colon \Delta\to\mathbb{R}$
in the triangle $\Delta$
is the linear function $I_{\Delta}u\colon \Delta\to\mathbb{R}$
\mscomm{which coincides with $u$ at the vertices of the triangle $\Delta$}.
\mscomm{The following well-known lemma is the main motivation for using cotangent weights.}

\begin{interpolation-lemma}\label{l-interpolation-1} 
\textup{(See \cite[\S4]{A1-BOB-duffin_1959})}
$
E_{{{T}}_\Delta}(u)=E_{{{S}}_\Delta}(I_{\Delta} u).
$
\end{interpolation-lemma}

Denote by $h$  and $\delta$ the maximal side length and the minimal angle of the triangle $\Delta$, respectively.

\begin{projection-lemma}\label{cl-projection}
For 
any vector $\overrightarrow v\in\mathbb{R}^2$ we have
$
|\overrightarrow v|\le 2\csc \delta \cdot \left(\frac{|\overrightarrow v\cdot\overrightarrow{xy}|}{|{xy}|}+
\frac{|\overrightarrow v\cdot\overrightarrow{yz}|}{|{yz}|}+
\frac{|\overrightarrow v\cdot\overrightarrow{zx}|}{|{zx}|}\right).
$
\end{projection-lemma}

\begin{proof}
Without loss of generality assume that
$\angle(\overrightarrow{xy},\overrightarrow{xz})=\delta$. 
Since $\angle(\overrightarrow{xy},\overrightarrow{xz})=
\pm\angle(\overrightarrow{v},\overrightarrow{xy})
\pm\angle(\overrightarrow{v},\overrightarrow{xz})$
it follows that at least one of the angles in the right-hand side, say, the first one, does not belong to the interval $(\frac{\pi}{2}-\frac{\delta}{2},\frac{\pi}{2}+\frac{\delta}{2})$.
Then
$
|\overrightarrow{v}|\le
\csc \frac{\delta}{2}\cdot\frac{|\overrightarrow v\cdot\overrightarrow{xy}|}{|{xy}|}\le
2\csc\delta\cdot\frac{|\overrightarrow v\cdot\overrightarrow{xy}|}{|{xy}|}.
$
\end{proof}

\begin{gradient-approximation-lemma}\label{cl-maximal-size}
$
\max_{w\in \Delta}|\nabla u(w)-\nabla I_{\Delta} u(w)|\le \mathrm{Const}_\delta\cdot h\cdot\max_{w\in \Delta}\|D^2 u(w)\|.
$
\end{gradient-approximation-lemma}

\begin{proof}
\mscomm{Denote $\overrightarrow{v}(w):=\nabla u(w)-\nabla I_{\Delta} u(w)$.}
By the Rolle theorem there is a point $w\in xy$ such that 
\mscomm{$\overrightarrow{v}(w)\cdot \overrightarrow{xy}/|\overrightarrow{xy}|=0$}.
Thus \mscomm{$|\overrightarrow{v}
\cdot \overrightarrow{xy}|/|\overrightarrow{xy}|\le \mathrm{Const}\cdot h\cdot \max_{\Delta}\|D^2 u\|$} in the triangle $\Delta$ because $\nabla I_{\Delta} u$ is a constant vector.
The same inequality holds with $xy$ replaced by $yz$ or $zx$. Thus the lemma follows from Projection Lemma~\ref{cl-projection}. 
\end{proof}


\begin{energy-approximation-lemma}\label{l-energy-approximation}
We have
\begin{equation*}
|E_{{{T}}_\Delta}(u)- E_{{{S}}_\Delta}(u)|\le
\mathrm{Const}_\delta\cdot \left(\max_{w\in \Delta}\|D^1 u(w)\|+h\max_{w\in \Delta}\|D^2u(w)\|\right)\cdot h\max_{w\in \Delta}\|D^2u(w)\|\cdot\mathrm{Area}(\Delta).
\end{equation*}
\end{energy-approximation-lemma}

\begin{proof} 
By Interpolation Lemma~\ref{l-interpolation-1} we get
\begin{equation*}
E_{{{T}}_\Delta}(u)- E_{{{S}}_\Delta}(u)
\mscomm{= 
\int_{\Delta}(|\nabla u|^2-|\nabla I_{\Delta} u|^2)\,dxdy}
= 
\int_{\Delta}(2\nabla u+\nabla I_{\Delta} u-\nabla u)(\nabla u-\nabla I_{\Delta} u)\,dxdy.
\end{equation*}
Thus by Gradient Approximation Lemma~\ref{cl-maximal-size} the lemma follows.
\end{proof}

\subsection{Convergence of energy in a cone}\label{ssec-energy-convergence-cone}

Now we are going to estimate the energy near a conical singularity $O$ of ${{S}}$. Let  ${{S}}_O$ be a neighborhood of the singular point $O$ bounded by a piecewise-geodesic closed broken line. Metrically ${{S}}_O$ is a part of a cone of aperture $2\pi/\gamma_O$  with the vertex $O$.
Introduce \mscomm{``}polar coordinates\mscomm{''} $(\rho,\phi)$ on ${{S}}_O$ with the origin at the vertex $O$. 

Let us give the proposed construction of a natural complex structure on
the polyhedral surface ${{S}}$ (see, e.g., \cite[Section~1.1.3]{Bobenko-11}). Define a chart $g_O\colon {{S}}_O\to\mathbb{C}$ by the formula $g_O(\rho,\phi):=\rho^{\gamma_O}\exp(i\gamma_O\phi)$. 
For a point $w\in{{S}}_O$, $w\ne O$, define a chart $g_{w}\colon N_w\to\mathbb{C}$ to be an orientation-preserving and distance-preserving map of a neighborhood $N_w\mscomm{\not\ni O}$ of the point $w$. \mscomm{E.g., for $w$ not belonging to the ray $\phi=2\pi/\gamma_O$ this map can be given by the formula $g_w(\rho,\phi):=\rho\exp(i\phi)$.} The constructed charts form a complex analytic atlas because all transition functions are compositions of maps of the form  $w\mapsto aw+b$ and $w\mapsto w^{\gamma_z}$ (away from the origin). The definition of an analytic function given in Section~\ref{ssec-convergence1} is compatible with the constructed complex structure by the singularity removal theorem. In fact any complex structure on a closed surface can be realized in this way by a piecewise flat metric \cite{Troyanov}. 

For a function $u\colon {{S}}_O\to\mathbb{R}$ denote $u_O:=u\circ g_O^{-1}$ and $u_{w}:=u\circ g_{w}^{-1}$. 
Throughout this subsection $u\colon \mathrm{Int}\,{{S}}_O\to\mathbb{R}$ is a function such that $u_O\colon \mathrm{Int}\, g_O{{S}}_O\to\mathbb{R}$ is smooth 
and smoothly extends to a neighborhood of $g_O{{S}}_O$. 
Denote by $|xy|$ the length of a geodesic segment $xy\subset {{S}}_O$. 
%

Let ${{T}}_{O}$ be a geodesic triangulation of ${{S}}_{O}$ with \mscomm{$O\in {{T}}_O^0$}, the maximal edge length $\le h$ and the minimal face angle $>\delta$.
Let
${{S}}_{O,h}$ be the union of faces of ${{T}}_{O}$
intersecting the $h$-neighborhood of the vertex $O$, and let ${{T}}_{O,h}$ the restriction of the triangulation ${{T}}_{O}$ to ${{S}}_{O,h}$.
We are going to estimate the difference between $E_{{{T}}_{O}}(u):=\sum_{\Delta\in{{T}}^2_{O}} E_{{{T}}_\Delta}(u)$ and $E_{{{S}}_{O}}(u) :=\sum_{\Delta\in{{T}}^2_{O}}E_{{{S}}_\Delta}(u)$.
The main difficulty in our estimates  is that 
the partial derivatives of $u_{w}$ are not necessarily bounded near the vertex~$O$. 


\begin{derivative-estimation-lemma}\label{l-derivatives}
For each $w=(\rho,\phi)\in{{S}}_O$ such that $w\ne O$
we have
$$
\mscomm{\left\|D^1u_w(z)\left|_{z=g_w(w)}\right.\right\|} \le \mathrm{Const}_{u,{{S}}_O} \cdot\rho^{\gamma_O-1}
\qquad\text{and}\qquad
\mscomm{\left\|D^2u_w(z)\left|_{z=g_w(w)}\right.\right\|} \le \mathrm{Const}_{u,{{S}}_O} \cdot\rho^{\gamma_O-2}.
$$
\end{derivative-estimation-lemma}

\begin{proof} By the chain rule and the Leibnitz rule we have
\mscomm{
\begin{align*}
\left\|D^1u_w(z)\left|_{z=g_w(w)}\right.\right\| &\le \mathrm{Const} \cdot  \left\|D^1(g_O\circ g^{-1}_w)(z)\left|_{z=g_w(w)}\right.\right\|\cdot \left\|D^1u_O(z)\left|_{z=g_O(w)}\right.\right\|;\\
\left\|D^2u_w(z)\left|_{z=g_w(w)}\right.\right\| &\le \mathrm{Const}\cdot 
\left\|D^2(g_O\circ g^{-1}_w)(z)\left|_{z=g_w(w)}\right.\right\|\cdot \left\|D^1u_O(z)\left|_{z=g_O(w)}\right.\right\| \\
&+\mathrm{Const}\cdot
\left\|D^1(g_O\circ g^{-1}_w)(z)\left|_{z=g_w(w)}\right.\right\|^2  \cdot\left\|D^2u_O(z)\left|_{z=g_O(w)}\right.\right\|.
\end{align*}
}
The map $g_O\circ g^{-1}_w\colon g_wN_w\to\mathbb{C}$ \mscomm{has the form $z\mapsto A z^{\gamma_O}$, where $|A|=1$}.
A direct computation shows that
$\|D^k(g_O\circ g^{-1}_w)\|\le \mathrm{Const}_{\gamma_O}\rho^{\gamma_O-k}$ for each $k=1,2$.
In particular,  $\|D^1(g_O\circ g^{-1}_w)\|^{\mscomm{2}}\le \mathrm{Const}_{\gamma_O}\rho^{2\gamma_O-2}\le \mathrm{Const}_{{{S}}_O}\rho^{\gamma_O-2}$.
Since $u_O\colon \mathrm{Int}\,g_O{{S}}_O\to\mathbb{R}$ is smooth and smoothly extends to a neighborhood of~$g_O{{S}}_O$
 it follows that $\|D^1 u_O\|, \|D^2 u_O\| \le\mathrm{Const}_{u,{{S}}_O}$
  on the compact set $g_O{{S}}_O$,
and the required estimate follows.
\end{proof}


\begin{lemma}\label{l-energy-distance}
For each $\Delta\in{{T}}_{O}^2-{{T}}_{O,h}^2$ we have
$
|E_{{{T}}_\Delta}(u)- E_{{{S}}_\Delta}(u)|\le
\mathrm{Const}_{u,\delta,{{S}}_O}\cdot h\cdot
\int_ \Delta\rho^{2\gamma_O-2}\, d\rho d\phi.
$
\end{lemma}

\begin{proof} 
Let $z\in\Delta$ be the vertex closest to $O$.
Since 
 $\Delta\in{{T}}_{O}^2-{{T}}_{O,h}^2$ 
it follows that for each point $(\rho,\phi)\in \Delta$ \mscomm{we have}
$h\le |Oz|\le \rho\le |Oz|+h\le 2|Oz|$. 
Thus by Derivative Estimation Lemma~\ref{l-derivatives}
and Energy Approximation Lemma~\ref{l-energy-approximation}
we get
\begin{align*}
|E_{{{T}}_\Delta}(u)- E_{{{S}}_\Delta}(u)|&\le
\mathrm{Const}_{u,\delta,{{S}}_O}
(|Oz|^{\gamma_O-1}+h|Oz|^{\gamma_O-2}) \cdot h|Oz|^{\gamma_O-2}
\mathrm{Area}(\Delta)\\
&\le  
\mathrm{Const}_{u,\delta,{{S}}_O}\cdot h\cdot
\int_ \Delta\rho^{2\gamma_O-2}\, d\rho d\phi.\\[-1.3cm]
\end{align*}  
\end{proof}

\begin{lemma}\label{l-outside-neighborhood-energy}
We have
$ |E_{{{T}}_{O}-{{T}}_{O,h}}(u)-E_{{{S}}_{O}-{{S}}_{O,h}}(u)|\le
\mathrm{Const}_{u,\delta,{{S}}_O}\cdot
\begin{cases}
h, &\text{if }\gamma_O>1/2;\\
h|\log \mscomm{\frac{h}{\mathrm{Diam}({{S}}_O)}}|, &\text{if } \gamma_O=1/2;\\
h^{2\gamma_O}, &\text{if }\gamma_O<1/2.
\end{cases}
$
\end{lemma}

\begin{proof}
Summing the inequalities of Lemma~\ref{l-energy-distance}, enlarging the integration domain, and evaluating the integral we get 
\begin{align*}
 |E_{{{T}}_{O}-{{T}}_{O,h}}(u)-E_{{{S}}_{O}-{{S}}_{O,h}}(u)|
 &\le
 \sum_{\Delta\in \mathcal T_{O}^2-{{T}}_{O,h}}
|E_{{{T}}_\Delta}(u)-E_{{{S}}_\Delta}(u)|\\
&\le
\mathrm{Const}_{u,\delta,{{S}}_O}\cdot h\cdot
\int_{{{S}}_{O}-{{S}}_{O,h}}\rho^{2\gamma_O-2}\, d\rho d\phi\\
&\le 
\mathrm{Const}_{u,\delta,{{S}}_O}\cdot h\cdot
\int_{h\le \rho\le \mathrm{Diam}({{S}}_O), 0\le\phi\le 2\pi/\gamma_O}\rho^{2\gamma_O-2}\, d\rho d\phi\\
&\le 
\mathrm{Const}_{u,\delta,{{S}}_O}\cdot
\begin{cases}
h, &\text{if }\gamma_O>1/2;\\
h|\log \mscomm{\frac{h}{\mathrm{Diam}({{S}}_O)}}|, &\text{if } \gamma_O=1/2;\\
h^{2\gamma_O}, &\text{if }\gamma_O<1/2.
\end{cases}
\\[-1.3cm]
\end{align*}
\end{proof}

\begin{remark}\label{rem-smaller-polygon}
Lemma~\ref{l-outside-neighborhood-energy} and its proof remain true, if we replace ${{S}}_O$ and ${{T}}_O$ in the left-hand side by a triangulated subset ${{S}}_O'\subset{{S}}_O$ with a triangulation ${{T}}_O'$ of the maximal edge length $\le h$ and the minimal face angle $>\delta$ (the constant in right-hand side still depends on ${{S}}_O$ but not ${{S}}_O'$). 
\end{remark}

\begin{lemma}\label{l-continuous-neighborhood-energy}
We have
$
 E_{{{S}}_{O,h}}(u)\le
\mathrm{Const}_{u,{{S}}_O}\cdot h^{2\gamma_O}.
$
\end{lemma}

\begin{proof}
By Derivative Estimation Lemma~\ref{l-derivatives} we have
\begin{equation*}
E_{{{S}}_{O,h}}(u)=
\sum_{\Delta\in{{T}}_{O,h}^2}\int_{\Delta} |\nabla u|^2 dxdy
\le\mathrm{Const}_{u,{{S}}_O}\int_{0\le\rho\le 2h,0\le \phi\le 2\pi/\gamma_O} \rho^{2\gamma_O-1}\, d\rho d\phi
 \le \mathrm{Const}_{u,{{S}}_O}\cdot h^{2\gamma_O}\\[-0.8cm].
\end{equation*}
\end{proof}

For an analogous estimate of the discrete energy we need the following auxiliary assertions. 

\begin{lemma}\label{l-difference} For each edge $xy\in{{T}}^1_O$ we have 
$|u(x)-u(y)|\le \mathrm{Const}_{u,{{S}}_O} \cdot |xy|\cdot \max\{|Ox|, |Oy|\}^{\gamma_O-1}.$
\end{lemma}

\begin{proof} 
Assume $|Ox|\le |Oy|$ without loss of generality.
Let a point $w=(\rho,\phi)$ move along the segment $xy$. Denote by $Oz$ the perpendicular from the point $O$ to the geodesic containing $xy$. 
By Derivative Estimation Lemma~\ref{l-derivatives}, the Bernoulli inequality, \mscomm{the inequality $\phi\le 2\sin\phi$ for $0\le\phi\le\pi/2$}, and the sine theorem for the triangle $Oxy$ (
in the development of the cone ${{S}}_O$ into the plane
) we have
\begin{align*} 
|u(x)-u(y)|&\le \int_{xy}\left\|D^1u_{w}(z)\left|_{z=g_{w}(w)}\right.\right\|\cdot (|d\rho|+\rho |d\phi|)\\
&\le \mathrm{Const}_{u,{{S}}_O}\int_{xy} \rho^{\gamma_O-1} (|d\rho|+\rho |d\phi|)\\
&\le \mathrm{Const}_{u,{{S}}_O}
\begin{cases}
\int_{xy} \rho^{\gamma_O-1} d\rho+|\int_{xy} \rho^{\gamma_O}d\phi|, &\text{if $x$ is between $y$ and $z$};\\
\int_{zx} \rho^{\gamma_O-1} d\rho +\int_{zy} \rho^{\gamma_O-1} d\rho+|\int_{xy} \rho^{\gamma_O}d\phi|, &\text{if $z$ is between $x$ and $y$};
\end{cases}
\\
&\le \mathrm{Const}_{u,{{S}}_O}\cdot
\left(\mscomm{2}|Oy|^{\gamma_O}-\mscomm{2}\max\{0,|Oy|-|xy|\}^{\gamma_O}+|Oy|^{\gamma_O}\cdot\angle xOy\right)\\
&\le
\mathrm{Const}_{u,{{S}}_O} \cdot
\begin{cases}
\mscomm{2}|Oy|^{\gamma_O}+|Oy|^{\gamma_O}\cdot \pi, &\text{if }|xy|\ge|Oy|;\\
\mscomm{2}\max\{1,\gamma_O\}\cdot |xy|\cdot |Oy|^{\gamma_O-1}+2|Oy|^{\gamma_O}\sin xOy, &\text{if }|xy|<|Oy|;
\end{cases}
\\
&\le\mathrm{Const}_{u,{{S}}_O} \cdot |xy|\cdot |Oy|^{\gamma_O-1}
\end{align*}
\mscomm{Here we set $\angle xOy=0$, if the segment $Oy$ contains $x$.}
\end{proof}



\begin{lemma}\label{l-min-angle-estimate} For each $xyz\in{{T}}^2_O$ we have
\begin{align*}
|xy|^2 &\le \mathrm{Const}_\delta\cdot\mathrm{Area}(xyz);\\ 
\max\{|Ox|, |Oy|, |Oz|\}&\le \mathrm{Const}_\delta\cdot\max\{|Ox|, |Oy|\}
\end{align*}
\end{lemma}

\begin{proof} Take the point $w$ such that $\angle xyw=\angle yxw=\delta$ and the triangles $xyz$ and $xyw$ have the same orientation. Since $\angle xyz,\angle yxz>\delta$ it follows that $xyw\subset yxz$. Thus $\mathrm{Area}(xyz)>\mathrm{Area}(xyw)=\tan\delta\cdot 
|xy|^2/4$, and the first inequality follows. 
Assume that $|Ox|\le |Oy|\le |Oz|$ without loss of generality. By the triangle inequality and the sine theorem we get the second inequality
\begin{equation*}
|Oz|-|Oy|\le |yz|
\le \csc \delta \cdot |xy|
\le \csc \delta \cdot (|Ox|+|Oy|)
\le 2\csc \delta\cdot |Oy|.\\[-0.6cm] 
\end{equation*}
\end{proof}

\begin{lemma}\label{l-energy-distance2}
For each $\Delta\in{{T}}_{O,h}^2$ we have
$
E_{{{T}}_\Delta}(u)\le
\mathrm{Const}_{u,\delta,{{S}}_O}\cdot
\mscomm{\int_{\Delta} {\rho}^{2\gamma_O-1} d\rho d\phi.}
$
\end{lemma}

\begin{proof}
Denote by $x,y,z$ the vertices of $\Delta$ so that $|Ox|\le |Oy|\le |Oz|$.
Since the minimal face angle is greater than $\delta$, using Lemmas~\ref{l-difference}--\ref{l-min-angle-estimate} and estimating the sum via the integral we get
\begin{align*}
   E_{{{T}}_\Delta}(u)&=
  \frac{1}{2}\cot xyz \cdot (u(z)-u(x))^2+\frac{1}{2}\cot yzx \cdot (u(x)-u(y))^2+\frac{1}{2}\cot zxy \cdot (u(y)-u(z))^2 \\
&\le\mathrm{Const}_{\delta}\left( (u(z)-u(x))^2+(u(x)-u(y))^2+(u(y)-u(z))^2\right)\\
&\le\mathrm{Const}_{u,\delta,{{S}}_O} 
\left(|Oz|^{2\gamma_O-2}\cdot |zx|^2+|Oy|^{2\gamma_O-2}\cdot |xy|^2+|Oz|^{2\gamma_O-2}\cdot |yz|^2
\right)
\\
&\le \mathrm{Const}_{u,\delta,{{S}}_O} \cdot|Oz| ^{2\gamma_O-2}\cdot\mathrm{Area}(\Delta)\\
&\le \mathrm{Const}_{u,\delta,{{S}}_O}\cdot 
\mscomm{\int_{\Delta} {\rho}^{2\gamma_O-1}\, d\rho d\phi. }
\end{align*}
\mscomm{The latter inequality is straightforward for $\gamma_O\le 1$ and requires the following estimate for $\gamma_O>1$:
$$
|Oz| ^{2\gamma_O-2}\cdot\mathrm{Area}(\Delta)
= 16 |Oz| ^{2\gamma_O-2}\cdot\mathrm{Area}(\frac{1}{4}\Delta)
\le 16 \int_{\frac{1}{4}\Delta} {(2\rho)}^{2\gamma_O-2}\, \rho\, d\rho d\phi
\le 
2^{2\gamma_O+2}\int_{\Delta} {\rho}^{2\gamma_O-1}\, d\rho d\phi,
$$
where $\frac{1}{4}\Delta\subset\Delta$ is obtained from $\Delta$ by the homothety with the center $z$ and the coefficient $1/4$.
}
\end{proof}

\begin{lemma}\label{l-discrete-neighborhood-energy}
$
  E_{{{T}}_{O,h}}(u)\le
\mathrm{Const}_{u,\delta,{{S}}_O}\cdot h^{2\gamma_O}.
$
\end{lemma}

\begin{proof}
\mscomm{Summing} the inequalities of Lemma~\ref{l-energy-distance2}, enlarging the integration domain, and evaluating the integral  we get
\begin{align*}
   E_{{{T}}_{O,h}}(u)&=
   \sum_{\Delta\in {{T}}_{O,h}^2}E_{{{T}}_\Delta}(u)\\
   &\le \mathrm{Const}_{u,\delta,{{S}}_O}\cdot 
   \int_{{{S}}_{O,h}} {\rho}^{2\gamma_O-1}\, d\rho d\phi\\
   &\le\mathrm{Const}_{u,\delta,{{S}}_O}\cdot
  \int_{0\le \rho\le 2h,0\le \phi\le 2\pi/\gamma_O} {\rho}^{2\gamma_O-1}\, d\rho d\phi\\
&\le\mathrm{Const}_{u,\delta,{{S}}_O}\cdot h^{2\gamma_O}.\\[-1.3cm]
\end{align*}
\end{proof}

\begin{example}\label{ex-counterexample}
The assumption on the minimal face angle in Lemma~\ref{l-energy-distance2} and Energy Convergence Lemma~\ref{l-energy-convergence} cannot be dropped. Indeed, take $0<\epsilon<1/2$,  $\gamma_O:=1/2-\epsilon$, and $u(\rho,\phi):=\rho^{\gamma_O}\cos(\gamma_O\phi)$. Consider the triangle $\Delta$ with the vertices $O$, $x:=(1/n, 0)$, and $y:=(1/n^{\epsilon},\pi/2)$. 
We have $\angle Oyx\to 0$ and
$E_{{{T}}_\Delta}(u)\ge 
\cot Oyx \cdot (u(x)-u(O))^2/2 = 
n^{1-\epsilon-2\gamma_O}/2 \to \infty$
as $n\to \infty$.
Let ${{T}}_O$ be an arbitrary triangulation of ${{S}}_O$ containing the triangle $\Delta$ and having maximal edge length $<2/n^{\epsilon}$. Then the inequalities of Lemmas~\ref{l-energy-distance2} and~\ref{l-energy-convergence} cannot hold.
\end{example}

\subsection{Convergence of energy in polyhedral surfaces}\label{ssec-energy-convergence}

Denote $\lambda_{{{S}}}(h):=
\begin{cases}
h, &\text{if }\gamma_{{S}}>1/2;\\
h|\log h|, &\text{if } \gamma_{{S}}=1/2;\\
h^{2\gamma_{{S}}}, &\text{if }\gamma_{{S}}<1/2.
\end{cases}$

\begin{proof}[Proof of  Energy Convergence Lemma~\ref{l-energy-convergence}] First take an auxiliary triangulation ${{T}}_1$ of the surface ${{S}}$ 
and set ${{S}}_z$ to be the union of faces containing a vertex $z\in {{T}}_1^0$.
Denote by $t{{S}}_{z}\subset {{S}}_{z}$ the set obtained from the cone ${{S}}_{z}$ by the homothety with the center $z\in {{T}}_1^0$ and the coefficient $0<t<1$.
Let $\mathrm{const}_{u,\delta,{{S}}}$
be $1/10$ of the minimal face height of ${{T}}_1$. 
Now take a triangulation ${{T}}$ of maximal edge length $h<\mathrm{const}_{u,\delta,{{S}}}$.
Let ${{S}}_z'\subset {{S}}_z$ be the union of faces of ${{T}}$ intersecting the set $\frac{8}{10}{{S}}_z$ \mscomm{and ${{T}}_z'$ be the restriction of ${{T}}$ to ${{S}}_z'$}. Then ${{S}}_z'\subset \frac{9}{10}{{S}}_z$ and $u_z\colon g_z(\frac{9}{10}{{S}}_z)\to \mathbb{R}$ smoothly extends to a neighborhood of $g_z(\frac{9}{10}{{S}}_z)$.
On the other hand, ${{S}}_z'\supset \frac{8}{10}{{S}}_z$ and thus  $\cup_{z\in {{T}}_1^0}{{S}}_z'\supset {{S}}$.
\mscomm{Replace the sets ${{S}}_z'$ and ${{T}}_z'$ by their appropriate subsets forming \emph{decompositions} of ${{S}}$ and ${{T}}$, respectively.}
 Estimating the energy separately in the $h$-neighborhoods of the vertices of ${{T}}_1$ and outside them by
Lemmas~\ref{l-continuous-neighborhood-energy}, \ref{l-discrete-neighborhood-energy}, \ref{l-outside-neighborhood-energy}, Remark~\ref{rem-smaller-polygon}, \mscomm{and the inequality $\gamma_{{S}}\le 1$} we get
\begin{align*}
|E_{{{T}}}(u\left|_{\widetilde{{{T}}}^0}\right.)- E_{{S}}(u)|
&\le
\sum_{z\in {{T}}^1_0}\left(
E_{{{T}}_{z,h}}(u)
+
E_{{{S}}_{z,h}}(u) + 
|E_{{{T}}_{z}'-{{T}}_{z,h}}(u)- 
E_{{{S}}_{z}'-{{S}}_{z,h}}(u)|\right)\\
&\le
\sum_{z\in {{T}}^0_1} \mathrm{Const}_{u,\delta,{{S}}_z}\cdot 
\begin{cases}
h, &\text{if }\gamma_z>1/2;\\
h|\log h|, &\text{if } \gamma_z=1/2;\\
h^{2\gamma_z}, &\text{if }\gamma_z<1/2.
\end{cases}
\\ 
&\le \mathrm{Const}_{u,\delta,{{S}}}\cdot
\lambda_{{{S}}}(h).
\\[-1.4cm]
 \end{align*}
\end{proof}





\begin{lemma} \label{l-hard-inequality} For any $P\in\mathbb{R}^{2g}$ we have
$E_{{T}}(u_{{{T}},P})\ge E_{{S}}(u_{{{S}},P})$.
\end{lemma}

\begin{proof}
Let $I_{{T}} u_{{{T}},P}\colon \widetilde{{{S}}}\to\mathbb{R}$ be the linear interpolation of
$u_{{{T}},P}\colon \widetilde{{{T}}}^0\to\mathbb{R}$
 in each face of $\widetilde{{{T}}}$. By Interpolation Lemma~\ref{l-interpolation-1} we get $E_{{T}}(u_{{{T}},P})=E_{{S}}(I_{{T}} u_{{{T}},P})$ because the metric of the surface $\widetilde{{{S}}}$ may have  singularities only at the vertices of  $\widetilde{{{T}}}$. Since the function $I_{{T}} u_{{{T}},P}\colon \widetilde{{{S}}}\to\mathbb{R}$
is continuous by the continuous counterpart of Variational Principle~\ref{l-variational-principle5} the result follows.
\end{proof}


\begin{energy-convergence-lemma}\label{l-hard-convergence} 
For each  $\delta>0$ and each vector $P\in\mathbb{R}^{2g}$ 
there are constants $\mathrm{Const}_{P,\delta,{{S}}},\mathrm{const}_{P,\delta,{{S}}}>0$ 
such that for any triangulation ${{T}}$ of ${{S}}$  
with the maximal edge length $h<\mathrm{const}_{P,\delta,{{S}}}$ and the minimal face angle $>\delta$
we have
\begin{equation*}
|E_{{{T}}}(u_{{{T}},P})-E_{{S}}(u_{{{S}},P})|\le
\mathrm{Const}_{P,\delta,{{S}}}\cdot
\lambda_{{{S}}}(h).
\end{equation*}
\end{energy-convergence-lemma}

\begin{proof}[Proof of Energy Form Convergence Lemma~\ref{l-hard-convergence}] 
The lemma follows from the following sequence of estimates provided by Lemma~\ref{l-hard-inequality}, Variational Principle~\ref{l-variational-principle5}, and Lemma~\ref{l-energy-convergence}:
$$
0\le E_{{{T}}}(u_{{{T}},P})-E_{{S}}(u_{{{S}},P})\le E_{{{T}}}(u_{{{S}},P}\left|_{\widetilde{{{T}}}^0}\right.)-E_{{S}}(u_{{{S}},P})\le
\mathrm{Const}_{P,\delta,{{S}}}\cdot
\lambda_{{{S}}}(h).
\\[-0.9cm]
$$
\end{proof}

\begin{corollary}\label{l-bounded-energy}
Let $\{{{T}}_n\}$ be a nondegenerate uniform sequence of Delaunay triangulations of ${{S}}$ with maximal edge length approaching zero as $n\to\infty$.  Let $P_n\in\mathbb{R}^{2g}$ be a sequence of $2g$-dimensional real vectors converging to a vector $P\in\mathbb{R}^{2g}$.
Then $E_{{{T}}_n}(u_{{{T}}_n,P_n})\to E_{{{S}}}(u_{{{S}},P})$ as $n\to\infty$.
\end{corollary}

\begin{proof} Both $E_{{{T}}_n}(u_{{{T}}_n,P})$ and $E_{{{S}}}(u_{{{S}},P})$ are quadratic forms in  $P\in\mathbb{R}^{2g}$. Thus by Energy Form Convergence Lemma~\ref{l-hard-convergence} the result follows.
\end{proof}

\subsection{Convergence of period matrices}\label{ssec:Convergence of period matrices}

\begin{proof}[Proof of Convergence Theorem for Period Matrices~\ref{th-main}]
By Lemmas~\ref{l-relation2} and~\ref{l-relation3}
the energies $E_{{{T}}}(u_{{{T}},P})$ and $E_{{{S}}}(u_{{{S}},P})$ are quadratic forms in $P\in\mathbb{R}^{2g}$ with the block matrices 
$E_{{{T}}}$ and $E_{{{S}}}$, respectively.  
Thus by Energy Form Convergence Lemma~\ref{l-hard-convergence} for each 
$\delta>0$
there exist two constants $\mathrm{Const}_{\delta,{{S}}}, \mathrm{const}_{\delta,{{S}}}>0$ such that for any triangulation ${{T}}$ of ${{S}}$ with maximal edge length
$h<\mathrm{const}_{\delta,{{S}}}$ and the minimal face angle $>\delta$ we have
$\|E_{{{T}}}-E_{{S}}\|\le
\mathrm{Const}_{\delta,{{S}}}\cdot \lambda_{{{S}}}(h).
$
In particular, $\| (\Imaginary \Pi_{{{T}}^*})^{-1}-(\Imaginary \Pi_{{{S}}})^{-1}\|\le
\mathrm{Const}_{\delta,{{S}}}\cdot \lambda_{{{S}}}(h)$ for  $h<\mathrm{const}_{\delta,{{S}}}$. 
Hence there exist two new constants $\mathrm{Const}'_{\delta,{{S}}}, \mathrm{const}'_{\delta,{{S}}}>0$ 
such that $\| (\Imaginary \Pi_{{{T}}^*})\|\le\mathrm{Const}'_{\delta,{{S}}}$ for  $h<\mathrm{const}'_{\delta,{{S}}}$. Assume further that $h<\min\{\mathrm{const}'_{\delta,{{S}}},\mathrm{const}_{\delta,{{S}}}\}$.

By the inequality $\|E_{{{T}}}-E_{{S}}\|\le
\mathrm{Const}_{\delta,{{S}}}\cdot \lambda_{{{S}}}(h)$ and the properties of the norm
we have 
\begin{align*}
\mathrm{Const}_{\delta,{{S}}}\cdot \lambda_{{{S}}}(h) &\ge \| (\Imaginary \Pi_{{{T}}^*})^{-1}\Real\Pi_{{{T}}} -(\Imaginary \Pi_{{{S}}})^{-1}\Real\Pi_{{{S}}} \| \\
&=
\| (\Imaginary \Pi_{{{T}}^*})^{-1}
(\Real\Pi_{{{T}}}-\Real\Pi_{{{S}}}) -
\left( (\Imaginary \Pi_{{{S}}})^{-1} - (\Imaginary\Pi_{{{T}}^*})^{-1}\right)\Real\Pi_{{{S}}} \| \\
&\ge 
\| (\Imaginary \Pi_{{{T}}^*})\|^{-1} \cdot
\|\Real\Pi_{{{T}}}-\Real\Pi_{{{S}}}\| -
\| (\Imaginary \Pi_{{{S}}})^{-1} - (\Imaginary\Pi_{{{T}}^*})^{-1}\|\cdot 
\|\Real\Pi_{{{S}}} \| \\
&\ge
(\mathrm{Const}'_{\delta,{{S}}})^{-1} \cdot
\|\Real\Pi_{{{T}}}-\Real\Pi_{{{S}}}\| -
\mathrm{Const}_{\delta,{{S}}} \cdot \lambda_{{{S}}}(h) \cdot
\|\Real\Pi_{{{S}}} \|
\end{align*}
Thus $\|\Real\Pi_{{{T}}}-\Real\Pi_{{{S}}}\| \le \mathrm{Const}''_{\delta,{{S}}}\cdot \lambda_{{{S}}}(h)$, where $\mathrm{Const}''_{\delta,{{S}}}:=
\mathrm{Const}'_{\delta,{{S}}}\cdot \mathrm{Const}_{\delta,{{S}}}\cdot(1+\|\Real\Pi_{{{S}}} \|)$.

Finally, since $\|E_{{{T}}}-E_{{S}}\|\le
\mathrm{Const}_{\delta,{{S}}}\cdot \lambda_{{{S}}}(h)$
it follows that 
$$
\|\Real \Pi_{{{T}}^*}(\Imaginary \Pi_{{{T}}^*})^{-1}\Real\Pi_{{T}}+\Imaginary\Pi_{{T}}-
\Real \Pi_{{{S}}}(\Imaginary \Pi_{{{S}}})^{-1}\Real\Pi_{{S}}-\Imaginary\Pi_{{S}}\|
\le \mathrm{Const}_{\delta,{{S}}}\cdot \lambda_{{{S}}}(h).
$$
Using the result of the previous paragraph and similar estimates we conclude that there exist a new constant $\mathrm{Const}'''_{\delta,{{S}}}$ such that
$
\|\Real \Pi_{{{T}}^*}(\Imaginary \Pi_{{{T}}^*})^{-1}\Real\Pi_{{T}}-
\Real \Pi_{{{S}}}(\Imaginary \Pi_{{{S}}})^{-1}\Real\Pi_{{S}}\|
\le \mathrm{Const}'''_{\delta,{{S}}}\cdot \lambda_{{{S}}}(h).
$
Thus 
$\|\Imaginary\Pi_{{{T}}}-\Imaginary\Pi_{{{S}}}\| \le (\mathrm{Const}'''_{\delta,{{S}}}+
\mathrm{Const}_{\delta,{{S}}})\cdot \lambda_{{{S}}}(h)$, which completes the proof of the theorem.
\end{proof}

\section{Convergence of discrete Abelian integrals}\label{sec-integrals}

In this section we prove Convergence Theorem for Abelian Integrals~\ref{th-Convergence of Abelian integrals}. 
The proof uses lemmas stated in the next two subsections.


\subsection{Equicontinuity}

Consider a Delaunay triangulation \mscomm{${{T}}$} of a polyhedral surface with boundary \mscomm{such that all faces are flat triangles}. A function 
\mscomm{$u\colon {{T}}^0\to\mathbb{R}$}
is \emph{discrete harmonic}, if it satisfies equation~\eqref{eq-def-harmonic} at each nonboundary vertex~$z$. \mscomm{Denote $E'_{{{T}}}(u):=\sum_{xy\in {{T}}^1, xy\not\in\partial {{T}}}{{c}}(xy)(u(x)-u(y))^2$, where the sum is over nonboundary edges.}
The \emph{eccentricity} of the triangulation is the infimum of the numbers $\mathrm{Const}$ such that the triangulation satisfies conditions~(D) (for each nonboundary edge) and~(U) from Section~\ref{ssec:Convergence of discrete Abelian integrals}. 
Denote by $e$ and $h'$ the eccentricity and twice the maximal circumradius of faces of the triangulation, respectively. 
To simplify the proofs below we need the following technical notion. We say that the triangulation has \emph{regular boundary}, if it extends to a Delaunay triangulation of a closed polyhedral surface such that the $h'$-neighborhood of the boundary of the initial surface contains no new singular points.

\begin{equicontinuity-lemma}\label{cl-log-inequality}
Let ${{S}}_O$ be the part of a cone with the vertex $O$ and the aperture $2\pi/\gamma_O$ bounded by a piecewise-geodesic broken line. Let ${{T}}_O$ be a Delaunay triangulation of ${{S}}_O$ with regular boundary such that $O\in{{T}}^0_O$.
Let $u\colon {{T}}_O^0\to \mathbb{R}$ be a discrete harmonic function.
Let $z,w\in {{T}}_O^0$ be two vertices \mscomm{at the distance $|zw|\ge h'$}.
Take $r$ such that $3\gamma_O'|zw|<r<\mathrm{Dist}(zw,\partial {{S}}_O)$, where $\gamma_O'=\max\{\gamma_O^2,1/\gamma_O^2\}$.
Then there is a constant $\mathrm{Const}_{e,\gamma_O}$  
such that 
\begin{equation}\label{eq-log-inequality}
\left|u(z)-u(w)\right|\le
\mathrm{Const}_{e,\gamma_O}\cdot E\mscomm{'}_{\!{{T}}_O}(u)^{1/2}\cdot \log^{-1/2}\frac{r}{3\gamma_O'|zw|},
\end{equation}
For $|zw|< h'< r/3\gamma_O'$ the same inequality holds with $|zw|$ replaced by $h'$.
\end{equicontinuity-lemma}

\begin{proof} 
The lemma is proved analogously to a similar estimate for \emph{quadrilateral lattices} in the plane \cite[Equicontinuity Lemma~2.4]{A1-Skopenkov}. Actually 
one can derive the lemma from that estimate as follows.
The used notions and results can be found in \cite[\S1, Remarks~3.4 and~4.8]{A1-Skopenkov} and Section~\ref{sec:Delaunay--Voronoi quadrangulation} below.

For each face of ${{T}}_O$, draw $3$ circumradii from the circumcenter to the vertices; see Figure~\ref{fig:2triangles}. 
Since the triangulation ${{T}}_O$ has regular boundary, by Theorem~\ref{l-transformation} it follows that the drawn segments do not have common interior points.
Erase hanging edges from the obtained graph. We get an orthogonal quadrilateral lattice ${Q}$ \mscomm{(because ${{T}}$ must have more than one face)}. Its twice maximal edge length is $h'$ and its eccentricity 
is~$\mathrm{Const}_e$. 
Extend the function $u\colon {{T}}_O^0\to\mathbb{R}$ to $Q^0$  by identical zero at ${Q}^0-{{T}}_O^0$.  The obtained function ${Q}^0\to \mathbb{R}$ is discrete harmonic and has the energy $E\mscomm{'}_{\!{{T}}_O}(u)$. 

Let the map $q\colon {{S}}_O\to \mathbb{R}^2$ be given in polar coordinates by the formula 
$q\colon (\rho,\phi)\mapsto (\sqrt{\gamma_O'}\rho,\gamma_O\phi)$. 
Identify each face of $Q$ (respectively, its image under $q$) with a quadrilateral $z_1z_2z_3z_4\subset\mathbb{C}$ by an orientation-preserving isometry (respectively, composed with $q^{-1}$). This makes $q(Q)$ a quad-surface 
and $u\circ q^{-1}\colon q(Q^0)\to \mathbb{R}$ a {discrete harmonic function}. 
Since the map $q\colon {{S}}_O\to \mathbb{R}^2$ increases distances by a factor at most $\gamma'_O$ it follows that $q(Q)$ has twice the maximal edge length at most $\gamma_O' {h}'$ and the eccentricity 
at most $\mathrm{Const}_{\gamma_O,e}$. 

Draw the square $R$ of side length $r>3\gamma'_O|zw|\ge 3|q(z)-q(w)|$ with the center at the midpoint of the segment $q(z)q(w)$ and the sides parallel and orthogonal to $q(z)q(w)$. Since $r<\mathrm{Dist}(zw,\partial {{S}}_O)\le \mathrm{Dist}(q(z)q(w),q(\partial {{S}}_O))$ it follows that $R\cap q(\partial{{S}}_O)=\emptyset$. 
By
\cite[Equicontinuity Lemma~2.4, Remarks~4.8 and~4.9]{A1-Skopenkov}  the lemma follows.
\end{proof}

\subsection{Harmonicity of a uniform limit}

A sequence of triangulated polygons $\{{{T}}_n\}$ \emph{approximates} a domain $\Omega\subset\mathbb{C}$, if for $n\to\infty$:
\begin{itemize}
\item the maximal distance from a point of $\partial {{T}}_n$ to the set $\partial\Omega$ tends to zero;
\item the maximal distance from a point of $\partial\Omega$  to the set $\partial {{T}}_n$ tends to zero;
\item the maximal edge length of the triangulation ${{T}}_n$ tends to zero.
\end{itemize}

\begin{lemma} \label{cl-harmonic-limit}
Let $\{{{T}}_n\}$ be a nondegenerate uniform sequence of Delaunay triangulations of polygons with regular boundary approximating a domain $\Omega\subset\mathbb{C}$. Let $u_n\colon {{T}}^0_n\to\mathbb{R}$ be a sequence of discrete harmonic functions uniformly converging to a continuous function $u\colon \Omega\to\mathbb{R}$. Then the function $u\colon \Omega\to\mathbb{R}$ is harmonic.
\end{lemma}

This lemma is proved analogously to a similar result for quadrilateral lattices \cite[Lemma~4.13 and the remark at the end of \S4.5]{A1-Skopenkov}. Actually it immediately reduces to that result analogously to the second paragraph of the proof of Equicontinuity Lemma~\ref{cl-log-inequality} above. 

\subsection{Convergence of multi-valued discrete harmonic functions} \label{ssec:Convergence of discrete harmonic functions}


We say that the function $u_{{{T}},P}\colon\widetilde {{{T}}}^0\to\mathbb{R}$ (respectively, $u_{{{S}},P}\colon\widetilde {{{S}}}\to\mathbb{R}$) 
is \emph{normalized} at a point $z$ if $u_{{{T}},P}(z)=0$
(respectively, $u_{{{S}},P}(z)=0$).


\begin{convergence-multi-valued-discrete-harmonic- functions}
\label{th-Convergence of Abelian integrals of the 1st kind}
Let $\{{{T}}_n\}$ be a nondegenerate uniform sequence of Delaunay triangulations of ${{S}}$ with maximal edge length approaching zero as $n\to\infty$. Let $z_n\in\widetilde{{{T}}}_n^0$ be a sequence of vertices converging to a point $z_0\in\widetilde{{{S}}}$.  Let $P_n\in\mathbb{R}^{2g}$ be a sequence of 
vectors converging to a vector $P\in\mathbb{R}^{2g}$.
Then the function $u_{{{T}}_n,P_n}\colon \widetilde{{{T}}}_n^0\to\mathbb{R}$  normalized at $z_n$ converges to  $u_{{{S}},P}\colon \widetilde{{{S}}}\to\mathbb{R}$ normalized at $z_0$
uniformly on each compact subset.
\end{convergence-multi-valued-discrete-harmonic- functions}

The proof goes along the lines of the proof of \cite[Convergence Theorem~1.2]{A1-Skopenkov} with several improvements required by the presence of curvature. We need an auxiliary lemma.

\begin{lemma}\label{l-energy-bound3} 
Let ${{T}}$ be a triangulation of a polyhedral surface with boundary such that all face angles are $>\delta$. Then 
for each $u\colon {{T}}^0\to\mathbb{R}$ we have 
$
E'_{{{T}}}(u)\le \mathrm{Const}_\delta\cdot E_{{{T}}}(u).
$
\end{lemma}

\begin{proof} Take a face $\Delta=xyz\in {{T}}^2$ such that $xz\subset\partial {{T}}$ and $\angle xyz>\pi/2$. By Lemmas~\ref{l-interpolation-1} and~\ref{l-min-angle-estimate} we get
$$
\frac{1}{2}|\cot xyz|\cdot \left(u(x)-u(z)\right)^2
\le\mathrm{Const}_\delta \cdot |\nabla I_{\Delta}u|^2 \cdot |xz|^2
\le\mathrm{Const}_\delta \cdot E_{{{T}}_\Delta}(u).
$$
Summing these inequalities over all such faces $\Delta$ we get 
$E'_{{{T}}}(u)-E_{{{T}}}(u)\le \mathrm{Const}_\delta\cdot E_{{{T}}}(u)$. 
\end{proof}

\begin{proof}[Proof of Theorem~\ref{th-Convergence of Abelian integrals of the 1st kind}] 
%
Take an arbitrary subsequence $\{\widetilde{{{T}}}_{n_k}\}$ of the given sequence $\{\widetilde{{{T}}}_n\}$. For brevity denote $\widetilde{{{T}}}_k:=\widetilde{{{T}}}_{n_k}$. Take a sequence of compact sets $K_1\subset K_2\subset\dots \subset \widetilde{{{S}}}$ such that $\widetilde{{{S}}}=\cup_{j=1}^\infty K_j$. Assume that $K_1$ contains all the points of the converging sequence $\{z_k\}$.
Since the sequence $\{{{T}}_{n}\}$ is nondegenerate uniform it follows that the face angles are greater than some constant $\delta>0$. 


We start with individual estimates near each vertex $v$ of the first triangulation $\widetilde{{{T}}}_1$.
Let ${{S}}_v$ be the union of faces of  $\widetilde{{{T}}}_1$ containing a vertex $v\in \widetilde{{{T}}}_1^0$.
Denote by $10r$ the minimal face height of the triangulation ${{{T}}}_1$. 
Take $k_1$ such that for each $k>k_1$ the maximal edge length of $\widetilde{{{T}}}_k$ is less than $r\csc\delta /\gamma_v'$. Take $k>k_1$.
Let ${{S}}_k'$ be the union of all the faces
of the triangulation $\widetilde{{{T}}}_k$ intersecting the set $\frac{8}{10}{{S}}_{v}$. Then ${{S}}_k'\subset \frac{9}{10}{{S}}_{v}$ and thus ${{S}}_k'$ has regular boundary.
Let ${{T}}_k'$ be the restriction of the triangulation  $\widetilde{{{T}}}_k$ to ${{S}}_k'$.
Let us estimate the right-hand side of inequality~\eqref{eq-log-inequality} from Equicontinuity Lemma~\ref{cl-log-inequality} for  $u:=u_{{{T}}_k,P_k}\left|_{({{T}}_k')^0}\right.$ and  $z,w\in \frac{7}{10}{{S}}_{v}$. 
By Lemma~\ref{l-energy-bound3} and Corollary~\ref{l-bounded-energy} the sequence of energies $E\mscomm{'}_{\!{{T}}_k}(u_{{{T}}_k,P_k})$ is bounded.
Thus by Equicontinuity Lemma~\ref{cl-log-inequality} it follows that 
the function $u_{{{T}}_k,P_k}\left|_{\frac{7}{10}{{S}}_{v}\cap\widetilde{{{T}}}_k^0}\right.$ has \emph{uniformly bounded differences}, i.e., there is a constant $\mathrm{Const}_{{{S}},\{{{T}}_n\}}$ not depending on $k$, $z$, $w$ such that for each $k>k_1$ and $z,w\in \frac{7}{10}{{S}}_{v}\cap\widetilde{{{T}}}^0_k$ we have
$|u_{{{T}}_k,P_k}(z)-u_{{{T}}_k,P_k}(w)|<\mathrm{Const}_{{{S}},\{{{T}}_n\}}$. By the same lemma the sequence $\{u_{{{T}}_k,P_k}\left|_{\frac{7}{10}{{S}}_{v}\cap\widetilde{{{T}}}_k^0}\right.\}$ is \emph{equicontinuous},
i.~e., there is a positive function $\delta(\epsilon)$ not depending on $k$, $z$, $w$ such that for each $k>k_1$ and any points $z,w\in \frac{7}{10}{{S}}_{v} \cap\widetilde{{{T}}}^0_k$ \mscomm{at the distance $|zw|<\delta(\epsilon)$} we have $|u_{{{T}}_k,P_k}(z)-u_{{{T}}_k,P_k}(w)|<\epsilon$. 


Let us combine the obtained estimates. Since the compact set $K_1$ is contained in the union of finitely many sets of the form $\frac{7}{10}{{S}}_{v}$ it follows that the sequence
$\{u_{{{T}}_k,P_k}\left|_{K_1\cap\widetilde{{{T}}}_k^0}\right.\}$ also has uniformly bounded differences \mscomm{(with the bound depending also on $K_1$)} and is equicontinuous. Moreover, the sequence 
is uniformly bounded, because all $z_k\in K_1$ and $u_{{{T}}_k,P_k}(z_k)=0$.
Then by the Arzel\`a--Ascoli theorem 
it follows that there is a continuous function $u_1\colon K_1 \to \mathbb{R}$ and a subsequence $\{l_k\}$ of the sequence $1,2,\dots $ such that $l_1=k_1$ and $u_{{{T}}_l,P_l}$ converges to $u_1$ uniformly in $K_1$.

Proceed to the next compact set $K_2$. Analogously, there is a \mscomm{continuous} function $u_{2}\colon K_2 \to \mathbb{R}$ and a subsequence $\{m_k\}$ of the sequence $\{l_k\}$ such that $m_1=l_1$, $m_2=l_2$, and
$u_{{{T}}_m,P_m}$ converges to $u_{2}$ uniformly on $K_2$. Clearly, $u_1=u_{2}$ on $K_1$. Thus the extension can be continued, and eventually we get a continuous function $u\colon \widetilde{{{S}}} \to \mathbb{R}$ and a subsequence $\{p_k\}$ of the sequence $1,2,\dots$ such that
$u_{{{T}}_p,P_p}$ converges to $u$ uniformly on each compact subset of $\widetilde{{{S}}}$.

Clearly, $u\colon \widetilde{{{S}}}\to\mathbb{R}$ has the same periods $P$ as $u_{{{S}},P}$, and $u(z_0)=0$. 
For each edge $e\in \widetilde{{{T}}}_1^1$ denote by ${{S}}_e\subset\widetilde{{{S}}}$ the union of the edge $e$ (without the endpoints) and the interiors of the two faces of  $\widetilde{{{T}}}_1$ containing the edge $e$. Then each set ${{S}}_e$ is isometric to a planar domain.
By Lemma~\ref{cl-harmonic-limit} applied to each domain ${{S}}_e$
it follows that the limit function $u\colon \widetilde{{{S}}}\to\mathbb{R}$ is harmonic in $\widetilde{{{S}}}$ possibly except the vertices of the triangulation $\widetilde{{{T}}}_1$ (where the metric of $\widetilde{{{S}}}$ may have conical singularities). By the singularity removal theorem it follows that the continuous function $u\colon \widetilde{{{S}}}\to\mathbb{R}$ is harmonic in the whole surface $\widetilde{{{S}}}$.
Thus $u=u_{{{S}},P}$ normalized at $z_0$.

Since the limit function $u$ is unique, it follows that the initial sequence $u_{{{T}}_n,P_n}\colon \widetilde{{{T}}}^0_n\to\mathbb{R}$, not just the subsequence $u_{{{T}}_k,P_k}$, converges to $u_{{{S}},P}$ uniformly on each compact subset.
\end{proof}

\subsection{Convergence of discrete Abelian integrals of the first kind}\label{ssec:Convergence of discrete Abelian integrals of the first kind}

\begin{proof}[Proof of Convergence Theorem for Abelian Integrals~\ref{th-Convergence of Abelian integrals}]
Let $P_n,P\in\mathbb{R}^{2g}$ be the periods of the real parts $\mathrm{Re}\,\phi^l_{{{T}}_n}\colon \widetilde{{{T}}}_n^0\to\mathbb{R}$ and
$\mathrm{Re}\,\phi^l_{{{S}}}\colon \widetilde{{{S}}}\to\mathbb{R}$, respectively.
Then by Second Existence and Uniqueness Theorem~\ref{th-fund3} it follows that $\mathrm{Re}\,\phi^l_{{{T}}_n}=u_{{{T}}_n,P_n}$
and $\mathrm{Re}\,\phi^l_{{{S}}}=u_{{{S}},P}$.
By Convergence Theorem for Period Matrices~\ref{th-main} we have
$P_n\to P$ as $n\to\infty$. Thus by Theorem~\ref{th-Convergence of Abelian integrals of the 1st kind}
the real parts $\mathrm{Re}\,\phi^l_{{{T}}_n}\colon \widetilde{{{T}}}_n^0\to\mathbb{R}$ converge to
$\mathrm{Re}\,\phi^l_{{{S}}}\colon \widetilde{{{S}}}\to\mathbb{R}$ uniformly on each compact subset.  Convergence of the imaginary parts is proved analogously using Conjugate Functions Principle~\ref{th-energy-conjugation}.
\end{proof}

\section{Discrete Riemann--Roch theorem}\label{sec-riemann-roch}


\subsection{Discrete Abelian integrals of the second kind}

A \emph{discrete Abelian integral of the second kind} is an arbitrary multi-valued function $f=(\mathrm{Re}f\colon \widetilde{{{T}}}^0\to\mathbb{R},\mathrm{Im}f\colon \widetilde{{{T}}}^2\to\mathbb{R})$. If it does not satisfy equation~\eqref{eq-def-analytic2} for an edge $e\in \vec{{{T}}}^1$ then we say that $f$ has a \emph{simple pole} at the edge $e$. The value
$$
\mathrm{res}_{e}\, f:=\mathrm{Im}f(r_e)-\mathrm{Im}f(l_e)+\mscomm{{c}}(e)\mathrm{Re}f(h_e)-\mscomm{{c}}(e)\mathrm{Re}f(t_e)
$$
is called the \emph{residue} of $f$.
In the case when all the periods vanish then the discrete Abelian integral of the second kind is called a \emph{discrete meromorhic function}; it can be viewed as a pair of single-valued functions $(\mathrm{Re}f\colon  {{T}}^0\to\mathbb{R},\mathrm{Im}f\colon  {{T}}^2\to\mathbb{R})$.
In questions concerning continuous limit one should assume that under triangulation refinement the number of poles of discrete meromorphic functions is fixed.

\begin{3rd-existence-and-uniqueness-theorem}\label{cor-uniqueness-second-kind2}
For each edge $e\in \vec{{{T}}}^1$  there is a unique (up to constant) discrete Abelian integral $\phi_{e}=(\mathrm{Re}\,\phi_{e}\colon \widetilde{{{T}}}^0\to\mathbb{R},\mathrm{Im}\,\phi_{e}\colon \widetilde{{{T}}}^2\to\mathbb{R})$ of the second kind with vanishing A-periods, with the only pole at the edge $e$, and with $\mathrm{res}_{e}\, \phi_{e}=1$.
\end{3rd-existence-and-uniqueness-theorem}

This theorem is proved analogously to First Existence and Uniqueness Theorem~\ref{th-fund}. The only difference is that the value $\mathrm{res}_{e}\, \phi_{e}$ should be considered as a parameter instead of the values of 
A-periods.

\begin{lemma}\label{cor-B-periods2} 
The B-periods of $\phi_{e}=(\mathrm{Re}\,\phi_{e}\colon \widetilde{{{T}}}^0\to\mathbb{R},\mathrm{Im}\,\phi_{e}\colon \widetilde{{{T}}}^2\to\mathbb{R})$ are 
$$
B_l=
\mathrm{Re}\,\phi_\mathcal{T^*}^{l}(t_e)- \mathrm{Re}\,\phi_\mathcal{T^*}^{l}(h_e) +i\mathrm{Re}\,\phi_{{T}}^l(t_e)-i\mathrm{Re}\,\phi_{{T}}^l(h_e),
\text{ where } l=1,\dots,g.
$$
\end{lemma}

\begin{proof} 
Let $A_1,\dots,A_g,B_1,\dots,B_g$ and $A_1',\dots,A_g',B_1',\dots,B_g'$  be the periods of $\phi_{e}$ and $\phi_{{T}}^{l}$, respectively.
All the A-periods $A_1,\dots,A_g,A_1',\dots,A_g'$ vanish except for $A'_l=1$. Using Riemann Bilinear Identity~\ref{th-Riemann2} and formula~\eqref{eq-def-analytic2}  we get
\begin{align*}
\mathrm{Im} B_l 
&=
\sum_{k=1}^g
(\mathrm{Re}A_k'\mathrm{Im}B_k-
\mathrm{Re}B_k'\mathrm{Im}A_k-
\mathrm{Re}A_k\mathrm{Im}B_k'+
\mathrm{Re}B_k\mathrm{Im}A_k')\\
&= \sum_{f\in{{T}}^1}
(\mathrm{Im}\,\phi_{e}(l_f)-\mathrm{Im}\,\phi_{e}(r_f))
(\mathrm{Re}\,\phi_{{T}}^{l}(h_f)-\mathrm{Re}\,\phi_{{T}}^{l}(t_f))\\
&-\sum_{f\in{{T}}^1}(\mathrm{Im}\,\phi_{{T}}^{l}(l_f)-\mathrm{Im}\,\phi_{{T}}^{l}(r_f))
(\mathrm{Re}\,\phi_{e}(h_f)-\mathrm{Re}\,\phi_{e}(t_f)) \\
&=\sum_{f\in{{T}}^1}\mathrm{res}_f \phi_{e}\cdot (\mathrm{Re}\,\phi_{{T}}^{l}(t_f)-\mathrm{Re}\,\phi_{{T}}^{l}(h_f))\\
&=\mathrm{Re}\,\phi_{{T}}^{l}(t_e)-\mathrm{Re}\,\phi_{{T}}^{l}(h_e),
\end{align*}
Analogously,  $\mathrm{Re}B_l=\mathrm{Re}\,\phi_\mathcal{T^*}^{l}(t_e)-\mathrm{Re}\,\phi_\mathcal{T^*}^{l}(h_e)$.
\end{proof}




\subsection{Discrete Abelian differentials}




The \emph{differential} of a discrete Abelian integral $f=(\mathrm{Re}f\colon\widetilde{{{T}}}^0\to\mathbb{R},\mathrm{Im}f\colon\widetilde{{{T}}}^2\to\mathbb{R})$ of the 1st kind is the function $df\colon \vec{{{T}}}^1\to\mathbb{R}$ given by the formula
\begin{equation}\label{eq-def-differential}
df(e):=\mathrm{Re}f(h_e)-\mathrm{Re}f(t_e)
\end{equation}
for each edge $e\in\vec{{{T}}}^1$. The functions $\omega\colon \vec{{{T}}}^1\to\mathbb{R}$ which are differentials of discrete Abelian integrals of the 1st kind are called \emph{discrete Abelian differentials of the 1st kind}. 

Abelian differentials of the 2nd kind have no discrete counterparts  in our setup.


A \emph{discrete Abelian differential of the third kind} is an arbitrary map $\omega\colon \vec{{{T}}}^1\to \mathbb{R}$ satisfying the condition $\omega(-e)=-\omega(e)$. Its \emph{integral} $\int_\gamma \omega$ along an oriented path $\gamma$ in the graph ${{T}}^1$ is the sum of values on the oriented edges of the path $\gamma$. The \emph{residue} $\mathrm{res}_{w}\omega$ of $\omega$ at a face $w\in {{T}}^2$ is the counterclockwise integral along the boundary of the face.
Given a collection of closed paths $\alpha_1,\dots,\alpha_g,\beta_1,\dots,\beta_g\colon [0;1]\to {{T}}^1$ in the graph ${{T}}^1$ forming a standard basis of the fundamental group, define the \emph{real periods} of $\omega$ as the integrals along
these paths. The real periods of $\omega$ may depend on the choice of the paths.

Similarly, a discrete Abelian differential $\omega\colon \vec{{{T}}}^1\to \mathbb{R}$ of the third kind can be integrated along any chain of oriented triangles such that the consecutive triangles in the chain induce opposite orientations on their unique common side. For example, if the chain consists of $2$ oriented triangles $xyz$ and $zyw$ then the integral of $\omega$ is by definition $\mscomm{{c}}(yz)\omega(yz)$. The details of the definition for a general chain are left to the reader. The \emph{residue} of $\omega$ at a vertex $z\in {{T}}^0$ is the number $$\mathrm{res}_{z}\omega:=i\sum_{e\in \vec{{{T}}}^1\,:\,h_e=z}\mscomm{{c}}(e)\omega(e).$$
The \emph{imaginary periods} of $\omega$ are defined analogously to the real ones as the integrals along certain fixed chains of triangles.

A direct checking shows that $\omega\colon \vec{{{T}}}^1\to \mathbb{R}$ is a discrete Abelian differential of the 1st kind if and only if its residues
at all faces and all vertices vanish. It is also easy to see that for a discrete Abelian differential of the 3rd kind the sum of the residues at all the vertices and all the faces vanishes.

\begin{4th-existence-and-uniqueness-theorem}\label{cor-uniqueness-third-kind2}
For any two distinct vertices $z,w\in {{T}}^0$  there is a unique discrete Abelian differential $d\phi_{z,w}\colon \vec{{{T}}}^1\to\mathbb{R}$ of the third kind with $\mathrm{res}_{z}d\phi_{z,w}=-\mathrm{res}_{w}d\phi_{z,w}=i$
and with all the other residues and both real and imaginary $A$-periods vanishing.
\end{4th-existence-and-uniqueness-theorem}

\begin{proof} \emph{Uniqueness.} Let $\omega,\omega'$ be two discrete Abelian differentials of the third kind as required. Then all residues of $\omega-\omega'$ vanish. Hence $\omega-\omega'$ is a discrete Abelian differential of the first kind. Since all A-periods of $\omega-\omega'$ vanish by Corollary~\ref{cor-zero-periods2}(2)$\Rightarrow$(1) it follows that $\omega=\omega'$.

\emph{Existence.} 
Let $\omega\colon \vec{{{T}}}^1\to \mathbb{R}$ be an unknown function satisfying the condition $\omega(-e)=-\omega(e)$. Consider its values at all the edges as variables, and the residue $ir$ at the vertex $z$ 
as a parameter.
For each vertex $y\in {{T}}^0$ distinct from $w$ and $z$ write one linear equation $\mathrm{res}_{y}\omega=0$. Write the linear equation $\mathrm{res}_{z}\omega=ir$. (The equality $\mathrm{res}_{w}\omega=-ir$ is obtained by summation of the written ones and therefore is not written.) For each face  $x\in {{T}}^2$ but one write one linear equation $\mathrm{res}_{x}\omega=0$. (Similar equation for the remaining face is obtained by summation of the written ones.)
For each $k=1,\dots,g$ write one linear equation $\int_{{\alpha}_k} \omega=0$ meaning that the real A-periods of $\omega$ vanish. Write $g$ analogous linear equations meaning that the imaginary A-periods of $\omega$ vanish. We get a system of $|{{T}}^0|+|{{T}}^2|+2g-2$ equations in $|{{T}}^1|$ variables. By the Euler formula it follows that the number of equations equals the number variables. By the uniqueness part of the theorem it follows that for $r=0$ the system has only trivial solution. Thus by the finite-dimensional Fredholm alternative it follows that for $r=1$ the system has a solution. This solution is the required discrete Abelian differential of the third kind.
\end{proof}

\begin{lemma}\label{cor-identity2} 
For any discrete meromorphic function $f=(\mathrm{Re}f\colon \widetilde{{{T}}}^0\to\mathbb{R},\mathrm{Im}f\colon \widetilde{{{T}}}^2\to\mathbb{R})$
and any  two distinct vertices $z,w\in {{T}}^0$ we have
$$
\mathrm{Re}f(z)-\mathrm{Re}f(w)=
\sum_{e\in {{{T}}}^1}d\phi_{z,w}(e)\cdot\mathrm{res}_{e}f
$$
\end{lemma}

\begin{proof} 
%
%
The lemma follows from the sequence of equalities
\begin{align*}
\sum_{e\in {{{T}}}^1}d\phi_{z,w}(e)\cdot\mathrm{res}_{e}f
&=
\sum_{e\in {{{T}}}^1}d\phi_{z,w}(e)\cdot 
(\mathrm{Im}f(r_e)-\mathrm{Im}f(l_e)+
\mscomm{{c}}(e)\mathrm{Re}f(h_e)-\mscomm{{c}}(e)\mathrm{Re}f(t_e))\\
&=
\sum_{x\in {{{T}}}^2}\mathrm{Im}f(x)
\sum_{e\in \vec{{{T}}}^1:r_e=x}d\phi_{z,w}(e)
+
\sum_{y\in {{{T}}}^0}\mathrm{Re}f(y)
\sum_{e\in \vec{{{T}}}^1:h_e=y}\mscomm{{c}}(e)d\phi_{z,w}(e)\\
&=
-\sum_{x\in {{T}}^2}
\mathrm{Im}f(x)\cdot\mathrm{res}_{x}d\phi_{z,w}
-
i\sum_{y\in{{T}}^0}
\mathrm{Re}f(y)\cdot\mathrm{res}_{y}d\phi_{z,w}\\
&=\mathrm{Re}f(z)-\mathrm{Re}f(w).\\[-1.6cm]
\end{align*}
\end{proof}

Similar results hold, if one takes two distinct \emph{faces} $z,w\in{{T}}^2$ instead of two distinct \emph{vertices}.

\subsection{Discrete Riemann--Roch theorem}

A \emph{divisor} on ${{T}}$ is an arbitrary map ${{T}}^0\sqcup {{T}}^1\sqcup {{T}}^2\to\{0,\pm1\}$.  The \emph{divisor} $(f)\colon {{T}}^0\sqcup {{T}}^1\sqcup {{T}}^2\to\{0,\pm1\}$ of a discrete meromorhic function $f=(\mathrm{Re}f\colon {{T}}^0\to\mathbb{R},\mathrm{Im}f\colon  {{T}}^2\to\mathbb{R})$ is defined by the formulae:
$$
(f)(z)=
\begin{cases}
+1, &\text{if }\mathrm{Re}f(z)=0;\\
0, &\text{if }\mathrm{Re}f(z)\ne 0;
\end{cases}
\qquad
(f)(e)=
\begin{cases}
-1, &\text{if }\mathrm{res}_{e}f\ne 0;\\
0, &\text{if }\mathrm{res}_{e}f=0;
\end{cases}
\qquad
(f)(w)=
\begin{cases}
+1, &\text{if }\mathrm{Im}f(w)=0;\\
0, &\text{if }\mathrm{Im}f(w)\ne 0;
\end{cases}
$$
for each $z\in {{T}}^0$, $e\in {{T}}^1$, $w\in {{T}}^2$.
The \emph{divisor} $(\omega)\colon {{T}}^0\sqcup {{T}}^1\sqcup {{T}}^2\to\{0,\pm1\}$ of a discrete Abelian differential $\omega\colon  \vec{{{T}}}^1\to\mathbb{R}$ of the third kind is defined by the formulae:
$$
(\omega)(z)=
\begin{cases}
-1, &\text{if }\mathrm{res}_{z}\omega\ne 0;\\
0,  &\text{if }\mathrm{res}_{z}\omega = 0;
\end{cases}
\qquad
(\omega)(e)=
\begin{cases}
+1, &\text{if }\omega(e)=0;\\
0,  &\text{if }\omega(e)\ne0;
\end{cases}
\qquad
(\omega)(w)=
\begin{cases}
-1, &\text{if }\mathrm{res}_{w}\omega\ne 0;\\
0,  &\text{if }\mathrm{res}_{w}\omega = 0.
\end{cases}
$$

\begin{remark} As opposed to the continuous theory, on a discrete Riemann surface divisors do \emph{not} form an Abelian group. This reflects the fact that the product of discrete analytic functions is not necessarily a discrete analytic function. Also the notions of multiple zeroes and multiple poles have no discrete counterparts in our setup.
\end{remark}

The \emph{degree} $\deg D$ of a divisor $D\colon {{T}}^0\sqcup {{T}}^1\sqcup {{T}}^2\to\{0,\pm1\}$ is the sum of all its values.
The divisor $D\colon {{T}}^0\sqcup {{T}}^1\sqcup {{T}}^2\to\{0,\pm1\}$ is \emph{greater or equal}
to a divisor $D'\colon {{T}}^0\sqcup {{T}}^1\sqcup {{T}}^2\to\{0,\pm1\}$, if
for each $z\in {{T}}^0\sqcup {{T}}^1\sqcup {{T}}^2$ we have $D(z)\ge D'(z)$. Denote by $l(D)$ the dimension of the space of discrete meromorphic functions on 
$ {{T}}$ with divisors greater or equal to the divisor $D$. Denote by $i(D)$ the dimension of the space of Abelian differentials of the third kind on $ {{T}}$ with divisors greater or equal to the divisor $D$. A divisor $D$ is \emph{admissible}, if $D(z)\le 0$ for each vertex $z\in {{T}}^0$, $D(e)\ge 0$
for each edge $e\in {{T}}^1$, and $D(w)\le 0$ for each face $w\in {{T}}^2$.

\begin{Riemann-Roch-theorem} \label{Riemann-Roch-theorem}
For any admissible divisor $D$ on a triangulated surface of genus $g$ we have
\begin{equation}\label{eq-riemann-roch}
l(-D)=\deg D-2g+2+i(D).
\end{equation}
\end{Riemann-Roch-theorem}

\begin{remark} This is a real version of the classical Riemann--Roch theorem: $
l_{{S}}(-D_{{S}})=\deg D_{{S}}-g+1+i(D_{{S}}),
$
where $D_{{S}}$ is a divisor on a Riemann surface ${{S}}$ and $l_{{S}}(-D_{{S}}),i_{{S}}(D_{{S}})$ are the \emph{complex} dimensions of continuous counterparts of the spaces from the definition of $l(-D),i(D)$. The discrete counterpart $D$ of a divisor $D_{{S}}$ 
has twice larger degree. An informal explanation is that a counterpart of one complex equation describing vanishing of a meromorphic function at a point is two real equations describing vanishing of the real and imaginary parts of a discrete meromorphic function.
\end{remark}

\begin{proof}[Proof of Riemann--Roch Theorem~\ref{Riemann-Roch-theorem}] Let $D\colon {{T}}^0\sqcup {{T}}^1\sqcup {{T}}^2\to\{0;\pm1\}$ be an admissible divisor. 
The \emph{support} of $D$ is the set  
$\mathrm{supp}D:=\{z\in  {{T}}^0\sqcup {{T}}^1\sqcup {{T}}^2:D(z)\ne 0 \}$. 
Denote $D_\infty=D\left|_{{{T}}^1}\right.$ and $D_0=D\left|_{{{T}}^0\sqcup {{T}}^2}\right.$.
Consider a discrete Abelian integral $f=(\mathrm{Re}f\colon {{T}}^0\to\mathbb{R},\mathrm{Im}f\colon  {{T}}^2\to\mathbb{R})$ of the 2nd kind with vanishing A-periods whose poles belong to the support $\mathrm{supp} D_\infty$. 
 By Corollary~\ref{cor-zero-periods2} it
 can be written in a unique way as
$$
f=\sum_{e\in\mathrm{supp}D_\infty}
\lambda_{e}\phi_{e}+\mathrm{const}
$$
with $\lambda_{e}=\mathrm{res}_e f\in \mathbb{R}$. By Lemma~\ref{cor-B-periods2}
its B-periods 
vanish if and only if for each $k=1,\dots,g$ we have
\begin{align}\label{eq-1st-condition2}
\sum_{e\in\mathrm{supp}D_\infty}
\lambda_{e}d\phi_{{{T}}}^k(e)&=0;\\
\label{eq-1st-condition3}\sum_{e\in\mathrm{supp}D_\infty}
\lambda_{e}d\phi_{{{T}}^*}^k(e)&=0.
\end{align}
Assume that equations~\eqref{eq-1st-condition2}--\eqref{eq-1st-condition3} hold.
Then by Lemma~\ref{cor-identity2} the restriction of $f=(\mathrm{Re}f\colon {{T}}^0\to\mathbb{R},\mathrm{Im}f\colon  {{T}}^2\to\mathbb{R})$ to the set $\mathrm{supp}D_0$ is a constant if and only if for each pair of elements $z,w\in {{T}}^0\cap\mathrm{supp}D_0$ or
$z,w\in {{T}}^2\cap\mathrm{supp}D_0$ we have
\begin{equation}\label{eq-2nd-condition2}
\sum_{e\in\mathrm{supp}D_\infty}
\lambda_{e}d\phi_{z,w}(e)=0.
\end{equation}
In what follows assume that equations~\eqref{eq-2nd-condition2} are written only for those pairs $(z,w)$ which form certain maximal trees in the two complete graphs on the sets ${{T}}^0\cap\mathrm{supp}D_0$ and ${{T}}^2\cap\mathrm{supp}D_0$, respectively. (The equations for the remaining pairs of vertices follow from the written ones.)

Consider \eqref{eq-1st-condition2}--\eqref{eq-2nd-condition2} as a system of linear equations in variables $\lambda_{e}$ for all $e\in \mathrm{supp}D_\infty$. By the previous paragraph the number
$l(-D)$ equals the dimension of the space of solutions of this system. Thus the rank of the system \eqref{eq-1st-condition2}--\eqref{eq-2nd-condition2} equals $\deg D_\infty-l(-D)$.

Transposing the matrix of the system \eqref{eq-1st-condition2}--\eqref{eq-2nd-condition2}, we get a new system of linear equations. Clearly, $i(D)$ equals the dimension of the space of solutions of the new system. Thus the rank of the new system equals $\deg D_0+2g-2-i(D)$. Since transposition does not change the rank we obtain the required identity
$\deg D_\infty-l(-D)=\deg D_0+2g-2-i(D)$.
\end{proof}

\begin{example} In the particular case of a triangulated  flat torus and a divisor supported by $1$ or $2$ edges Riemann--Roch Theorem~\ref{Riemann-Roch-theorem} implies the following: 
\begin{itemize}
\item there are no nonconstant discrete meromorphic functions with exactly one simple pole;
\item there exists a nonconstant discrete meromorphic function with  exactly two simple poles at two given edges if and only if the edges are parallel to each other.
\end{itemize}
This is proved by a direct computation of the right-hand side of~\eqref{eq-riemann-roch} using explicit construction of discrete Abelian integrals in Example~\ref{ex-torus}.
\end{example}

\section{Generalizations}\label{sec-generalization}

\subsection{Quad-surfaces}\label{ssec-quads}


Let us discuss a generalization of discrete Riemann surfaces considered in this paper above. This generalization has been introduced by Mercat \cite{MR2349680}. It has a physical interpretation in terms of alternating-current networks; cf. \cite[Section~5.2]{A1-Skopenkov}. 

Let ${{Q}}$ be a cell decomposition of the surface ${{S}}$ with quadrilateral faces. Assume that each vertex of ${{Q}}$ is paint either 
black or white so that the endpoints of each edge have different colors. Assume also that the intersection of any two faces is either empty, or a single vertex, or a single edge. A \emph{chart} on ${{Q}}$ is an orientation-preserving 
homeomorphism of a face and a (not necessarily convex) quadrilateral $z_1z_2z_3z_4$ in the complex plane $\mathbb{C}$. An \emph{atlas} on ${{Q}}$ is a collection of charts, one for each face  (no agreement of charts for distinct faces is assumed). A cell decomposition ${{Q}}$ with a fixed atlas is called a \emph{quad-surface}.

A function $f\colon{{Q}}^0\to\mathbb{C}$ is \emph{discrete analytic on} ${{Q}}$, if the difference quotients along the two diagonals of each face are equal, i.~e., 
for each face $z_1z_2z_3z_4\in{{Q}}^2$ we have 
\begin{equation}\label{eq-def-analytic}
\frac{f(z_1)-f(z_3)}{z_1-z_3}=
\frac{f(z_2)-f(z_4)}{z_2-z_4}.
\end{equation}

A \emph{multi-valued function on} ${{Q}}$ \emph{with periods} 
$\mathbf{A}_1,\dots,\mathbf{A}_{g},
\mathbf{B}_1,\dots,\mathbf{B}_{g},
\mathbb{A}_1,\dots,\mathbb{A}_{g},
\mathbb{B}_1,\dots,\mathbb{B}_{g}
\in\mathbb{C}$ is a function 
$f\colon\widetilde{{{Q}}}^0\to\mathbb{C}$ such that
for each $k=1,\dots,g$, each black vertex $z\in \widetilde{{{Q}}}^0$, and each white vertex $w\in \widetilde{{{Q}}}^0$ we have
\begin{align*}
f(d_{\alpha_k}z)-f(z)&=\mathbf{A}_k; & 
f(d_{\beta_k}z)-f(z)&=\mathbf{B}_k;\\
f(d_{\alpha_k}w)-f(w)&=\mathbb{A}_k; &
f(d_{\beta_k}w)-f(w)&=\mathbb{B}_k.
\end{align*}
The numbers $\mathbf{A}_1,\dots,\mathbf{A}_{g},
\mathbb{A}_1,\dots,\mathbb{A}_{g}$ 
are called the \emph{A-periods} 
of the multi-valued function $f$. 
A multi-valued discrete analytic function is called a \emph{discrete Abelian integral of the 1st kind} on ${{Q}}$.


The following theorem is proved analogously to Riemann Bilinear Identity~\ref{th-Riemann2}.

\begin{theorem}\label{th-riemann-identity}
For any two multi-valued functions $f,f'\colon \widetilde{{{Q}}}^0\to\mathbb{C}$
with periods
$$
\mathbf{A}_1,\dots,\mathbf{A}_{g},
\mathbf{B}_1,\dots,\mathbf{B}_{g},
\mathbb{A}_1,\dots,\mathbb{A}_{g},
\mathbb{B}_1,\dots,\mathbb{B}_{g}
\quad\text{and}\quad
\mathbf{A}'_1,\dots,\mathbf{A}'_{g},
\mathbf{B}'_1,\dots,\mathbf{B}'_{g},
\mathbb{A}'_1,\dots,\mathbb{A}'_{g},
\mathbb{B}'_1,\dots,\mathbb{B}'_{g},
$$ 
respectively, we have
\begin{multline*}
\sum_{z_1z_2z_3z_4\in {{Q}}^2}
\left((f(z_1)-f(z_3))(f'(z_2)-f'(z_4))-(f(z_2)-f(z_4))(f'(z_1)-f'(z_3))\right)\\
=\sum_{k=1}^g
(\mathbf{A}_k\mathbb{B}'_{k}-
\mathbb{B}_{k}\mathbf{A}'_k+
\mathbb{A}_{k}\mathbf{B}'_{k}-
\mathbf{B}_{k}\mathbb{A}'_{k}).
\end{multline*}
\end{theorem}

Analogously to First Existence and Uniqueness Theorem~\ref{th-fund} one can prove that there is a unique (up to constant) Abelian integral $\phi^l_{{Q}}\colon\widetilde{{{Q}}}^0\to\mathbb{C}$
 of the 1st kind whose A-periods are given by the formulae $\mathbf{A}_k=\delta_{kl}$, $\mathbb{A}_k=\delta_{kl}$,
where $k=1,\dots,g$. Let $\mathbf{B}^l_1,\dots,\mathbf{B}^l_{g},
\mathbb{B}^l_1,\dots,\mathbb{B}^l_{g}$ be the B-periods of~ $\phi^l_{{Q}}\colon\widetilde{{{Q}}}^0\to\mathbb{C}$. The matrix 
$(\Pi_{{Q}})_{kl}:=\frac{1}{2}(\mathbf{B}^l_k+\mathbb{B}^l_k)$ is the \emph{period matrix} of the quad-surface ${{Q}}$.

\begin{problem} Generalize Convergence Theorem for Period Matrices~\ref{th-main} to quad-surfaces.
\end{problem}

Convergence of period matrices for quad-surfaces cannot be proved similarly to triangulated surfaces. Difficulties appear in the proof of a counterpart of Lemma~\ref{l-hard-inequality} because for a quad-surface the natural ``interpolation'' $I_{{{Q}}} u$ from \cite[Section~5.1]{A1-Skopenkov} is not continuous.

\begin{problem} Which complex $g\times g$ matrices arise as period matrices of quad-surfaces of genus~$g$?
\end{problem}

This discrete counterpart of the Shottky problem is related to an inverse problem for alternating-current networks and seems to be very difficult.

\subsection{Delaunay--Voronoi quadrangulation}
\label{sec:Delaunay--Voronoi quadrangulation}

Let us show that the discretization of Riemann surfaces studied in Sections~\ref{sec-main}--\ref{sec-riemann-roch} is a particular case of the one from Section~\ref{ssec-quads} in a natural sense.
By the \emph{circumcenter} $w^*$ of a face $w$ of a Delaunay triangulation we mean the center of the immersed flat intrinsic disk containing the face $w$ in the interior and the vertices of $w$ in the boundary (such disk exists by the results of~\cite{Bobenko-Springborn-07}). By the \emph{circumradius} we mean a geodesic segment lying in the disk and joining the circumcenter with a point in the boundary.


\begin{theorem}\label{l-transformation} 
Let ${{T}}$ be a Delaunay triangulation of \mscomm{a closed polyhedral} surface~${{S}}$ satisfying condition (D) from Section~\ref{ssec:Convergence of discrete Abelian integrals}. For each face of ${{T}}$, draw $3$ circumradii from the circumcenter to the vertices; see Figure~\ref{fig:2triangles}. 
Then the drawn segments do not have common interior points and thus form a 
quad-surface $Q$, in which each face is identified with a planar quadrilateral 
by an orientation-preserving isometry. 
\end{theorem}

\begin{proof} 
Take an oriented edge $e\in\vec{{{T}}}^1$. Let $w\in{{T}}^2$ be a face containing $e$. By the results of \cite{Bobenko-Springborn-07} there is an orientation-preserving and \mscomm{locally} isometric immersion $g_w\colon D_w\to{{S}}$ of a disk $D_w\subset\mathbb{C}$ such that the face $w$ and the vertices of $w$ are in the interior \mscomm{and} the boundary of $g_wD_w$, respectively. 
Without loss of generality assume that $g_{r_e}^{-1}t_e=g_{l_e}^{-1}t_e$ and
$g_{r_e}^{-1}h_e=g_{l_e}^{-1}h_e$. Denote $z_1:=g_{r_e}^{-1}t_e$ and 
$z_3:=g_{r_e}^{-1}h_e$.
  Let $z_2:=g_{l_e}^{-1}l_e^*$ and $z_4:=g_{r_e}^{-1}r_e^*$ be the centers of $D_{l_e}$ and $D_{r_e}$, respectively. Define a \mscomm{locally} isometric immersion $g_e\colon z_1z_2z_3z_4\to{{S}}$ to be equal to $g_{l_e}$ in the triangle $z_1z_2z_3$ and to $g_{r_e}$ in the triangle $z_1z_3z_4$. 
The images $g_e(z_1z_2z_3z_4)$ are the faces of the future quad-surface $Q$.

Compare the sum of the areas of the quadrilaterals $z_1z_2z_3z_4$ with the area of the union of their images. Clearly, 
$\mathrm{Area}(z_1z_2z_3z_4)=|h_et_e|^2(\cot \alpha_e + \cot \beta_e)/4$ for each edge $e\in {{T}}^1$. 
Since ${{T}}$ is Delaunay and satisfies condition (D) 
it follows that these oriented areas are all positive. Expressing the area of each face of ${{T}}$ as an algebraic sum of three terms of the form $\mathrm{Area}(z_1z_2z_3)$ we get $\sum_{e\in {{T}}^1} \mathrm{Area}(z_1z_2z_3z_4)=\mathrm{Area}({{S}})$. 
Finally, 
$\bigcup_{e\in {{T}}^1} g_e(z_1z_2z_3z_4)\supset {{S}}$ because $\partial(\bigcup_{e\in {{T}}^1} g_e(z_1z_2z_3z_4))=\emptyset$. This implies that the maps $g_e\colon z_1z_2z_3z_4\to {{S}}$ neither self-intersect nor overlap with each other, hence the drawn segments do not have common interior points.
\end{proof}

Now it is easy to construct an ``isomorphism'' between the discretizations on~${{T}}$ and~$Q$. For two arbitrary functions $u\colon {{T}}^0\to\mathbb{R},v\colon{{T}}^2\to\mathbb{R}$
define a function $f\colon Q^0\to\mathbb{R}$ by the formula $f(z):=u(z)$ for each $z\in {{T}}^0$ and $f(w^*):=iv(w)$ for each $w\in {{T}}^2$. 
Then for each face $z_1z_2z_3z_4\in Q^2$ with $e:=z_1z_3\in {{T}}^1$ equations~\eqref{eq-def-analytic2} and~\eqref{eq-def-analytic} are equivalent. 

\begin{problem} Find a direct proof of the existence of the ``Delaunay--Voronoi quadrangulation'' $Q$ of a polyhedral surface as in Theorem~\ref{l-transformation} not using the existence of a Delaunay triangulation.
\end{problem}

\mscomm{

\subsection{Numerical experiments}\label{ssec-num}

Let us present numerical analysis of the above convergence results using a software  by S.~Tikhomirov.  Take the surface ${{{S}}}$ obtained 
by gluing $3$ unit squares 
as shown in Figure~\ref{fig-l} to the left. This surface has genus $g=2$, the parameter $\gamma_S=1/3$, 
and the period matrix $\Pi_{{S}}=\frac{i}{3}
\left(
\begin{matrix}
5 & -4\\
-4 & 5
\end{matrix}
\right)$
\cite[p.~220]{A1-BOB-BobenkoMercatSchmies}.
Take the triangulation ${{T}}_n$ of the surface ${{S}}$ shown in Figure~\ref{fig-l} to the right. Computing the discrete harmonic functions $u_{{{T}}_n,P}\colon \widetilde{{{T}}}_n^0\to\mathbb{R}$ for all basis vectors $P\in\mathbb{R}^{2g}$ (using classical method of relaxations), then the energy matrix $E_{{{T}}_n}$, and then the discrete period matrix $\Pi_{{{T}}_n}$ (using Lemma~\ref{l-relation2}) we get
the values in Table~\ref{t-num}. The values in the last column are nearly equal, which shows that theoretical approximation order in 
Theorem~\ref{th-main}
agrees with the experimental one. Table~\ref{t-num} corrects wrong values in \cite[Table in p.~220, 1st surface]{A1-BOB-BobenkoMercatSchmies}.

\begin{figure}[htbp]
\definecolor{qqqqff}{rgb}{0,0,1}
\begin{tikzpicture}[line cap=round,line join=round,>=triangle 45,x=1.8cm,y=1.8cm]
\clip(-4.61,-0.35) rectangle (0.81,2.33);
\draw [->] (-3.5,2) -- (-2.5,2);
\draw [->] (-4.5,1) -- (-3.5,1);
\draw [->] (-4.5,0) -- (-3.5,0);
\draw [->] (-3.5,0) -- (-2.5,0);
\draw (-3.5,1)-- (-2.5,1);
\draw (-3.5,1)-- (-3.5,0);
\draw (-1.5,0)-- (-1.5,1);
\draw (-1.5,1)-- (-0.5,1);
\draw (-0.5,1)-- (-0.5,2);
\draw (-0.5,2)-- (0.5,2);
\draw (0.5,2)-- (0.5,1);
\draw (0.5,1)-- (0.5,0);
\draw (-1.5,0)-- (-0.5,0);
\draw (-0.5,0)-- (0.5,0);
\draw (-0.5,1)-- (-0.5,0);
\draw (-0.5,1)-- (0.5,1);
\draw (-0.5,1.5)-- (0.5,1.5);
\draw (0,2)-- (0,1);
\draw (0,1)-- (0,0);
\draw (-0.5,0.5)-- (0.5,0.5);
\draw (-0.5,0.5)-- (-1.5,0.5);
\draw (-1,1)-- (-1,0);
\draw (-0.5,1.5)-- (0,2);
\draw (-1.5,0.5)-- (-1,1);
\draw (-1.5,0)-- (0.5,2);
\draw (-1,0)-- (0.5,1.5);
\draw (-0.5,0)-- (0.5,1);
\draw (0,0)-- (0.5,0.5);
\draw (-1.27,1.73) node[anchor=north west] {${{T}}_n$};
\draw (-4.1,1.77) node[anchor=north west] {${{S}}$};
\draw [->] (-3.5,2) -- (-3.5,1);
\draw [->] (-2.5,2) -- (-2.5,1);
\draw (-4.5,1)-- (-3.5,1);
\draw (-3.97,1) -- (-4,0.96);
\draw (-3.97,1) -- (-4,1.04);
\draw (-3.5,1)-- (-3.5,2);
\draw (-3.52,1.5) -- (-3.48,1.5);
\draw (-3.5,2)-- (-2.5,2);
\draw (-2.94,2) -- (-2.97,1.96);
\draw (-2.94,2) -- (-2.97,2.04);
\draw (-3,2) -- (-3.03,1.96);
\draw (-3,2) -- (-3.03,2.04);
\draw (-2.5,2)-- (-2.5,1);
\draw (-2.48,1.5) -- (-2.52,1.5);
\draw (-2.5,1)-- (-2.5,0);
\draw (-2.47,0.51) -- (-2.53,0.51);
\draw (-2.47,0.49) -- (-2.53,0.49);
\draw (-4.5,0)-- (-4.5,1);
\draw (-4.53,0.49) -- (-4.47,0.49);
\draw (-4.53,0.51) -- (-4.47,0.51);
\draw (-3.5,0)-- (-2.5,0);
\draw (-2.94,0) -- (-2.97,-0.04);
\draw (-2.94,0) -- (-2.97,0.04);
\draw (-3,0) -- (-3.03,-0.04);
\draw (-3,0) -- (-3.03,0.04);
\draw (-4.5,0)-- (-3.5,0);
\draw (-3.97,0) -- (-4,-0.04);
\draw (-3.97,0) -- (-4,0.04);
\draw [->] (-4.5,1) -- (-4.5,0);
\draw [->] (-2.5,1) -- (-2.5,0);
\begin{scriptsize}
\draw (-2.44,1.13) node[anchor=north west] {$\beta_2^{-1}$};
\draw (-3.55,-0.12) node[anchor=north west] {$\alpha_2$};
\fill [color=qqqqff] (-3.5,1) circle (1.5pt);
\fill [color=qqqqff] (-3.5,2) circle (1.5pt);
\fill [color=qqqqff] (-2.5,2) circle (1.5pt);
\fill [color=qqqqff] (-2.5,1) circle (1.5pt);
\fill [color=qqqqff] (-4.5,1) circle (1.5pt);
\draw[color=black] (-3.92,1.14) node {$\alpha_1$};
\fill [color=qqqqff] (-4.5,0) circle (1.5pt);
\fill [color=qqqqff] (-3.5,0) circle (1.5pt);
\fill [color=qqqqff] (-2.5,0) circle (1.5pt);
\fill [color=qqqqff] (-1.5,0) circle (1.5pt);
\fill [color=qqqqff] (-1.5,1) circle (1.5pt);
\fill [color=qqqqff] (-0.5,1) circle (1.5pt);
\fill [color=qqqqff] (-0.5,2) circle (1.5pt);
\fill [color=qqqqff] (0.5,2) circle (1.5pt);
\fill [color=qqqqff] (0.5,1) circle (1.5pt);
\draw[color=black] (0.72,1.15) node {$n$};
\fill [color=qqqqff] (0.5,0) circle (1.5pt);
\fill [color=qqqqff] (-0.5,0) circle (1.5pt);
\draw[color=black] (-0.39,-0.22) node {$n$};
\draw[color=black] (-3.32,1.55) node {$\beta_1$};
\end{scriptsize}
\end{tikzpicture}
\caption{The surface ${{S}}$ and its ``triangulation'' ${{T}}_n$ for $n=4$; see Section~\ref{ssec-num} for details. The paths $\alpha_1,\beta_1,\alpha_2,\beta_2\colon [0,1]\to {{S}}$ form a standard basis of the 1st homology (but not homotopy) group.}
\label{fig-l}
\end{figure}
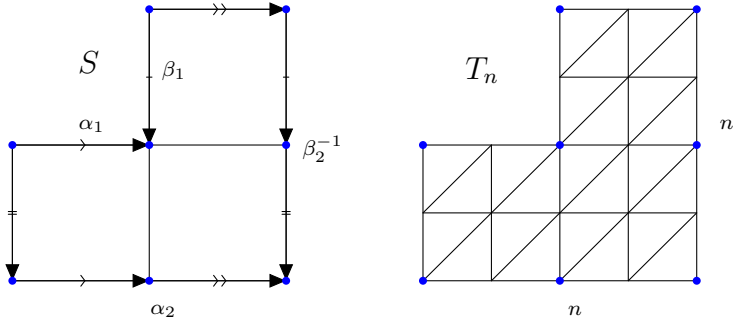

\begin{table}[htbp]
\mscomm{
\caption{Numerical experiments on approximation of period matrices by their discrete counterparts.}
\label{t-num}
\begin{center}
\begin{tabular}{|c|c|c|}
\hline
$n$  &  $\|\Pi_{{{T}}_n}-\Pi_{{{S}}}\|$ & $\|\Pi_{{{T}}_n}-\Pi_{{{S}}}\|\cdot h^{-2\gamma_S}$ \\
\hline
8     & 0.611 & 1.22 \\
16   & 0.363 & 1.15 \\
32   & 0.220 & 1.11 \\
64   & 0.136 & 1.08 \\
128 & 0.084 & 1.07 \\
256 & 0.053 & 1.06 \\
\hline
\end{tabular}
\end{center}
}
\end{table}


}

\subsection{Underwater reefs}

Let us warn the reader on possible dangers arising on the way of proving convergence results.

\begin{remark}\label{rem-error} The proof of convergence of discrete period matrices suggested in \cite[Proof of 3.1]{Mercat-02} is erroneous. It essentially relies on the assumption that the approximation of discrete analytic functions by discrete polynomials
constructed in \cite{Mercat-08} is \emph{uniform} while the approximation is in fact \emph{pointwise} and \emph{not} uniform by \cite[Section D]{Mercat-08}. 
\end{remark}


\subsection*{Acknowledgements}

The authors are grateful to D.~Chelkak, S. von Deylen, I.~Dynnikov, A.~Gaifullin, F.~G\"unther, S.~Lando, C.~Mercat, A.~Pakharev, M.~Wardetzky for useful discussions and to S.~Tikhomirov for writing a software for numerical experiments.



\bibliographystyle{plain}
\bibliography{A01}

\begin{thebibliography}{10}

\bibitem{MR2355607}
Matthew Baker and Serguei Norine.
\newblock Riemann-{R}och and {A}bel-{J}acobi theory on a finite graph.
\newblock {\em Adv. Math.}, 215(2):766--788, 2007.

\bibitem{Bobenko-11}
A.~I. Bobenko.
\newblock Introduction to compact {R}iemann surfaces.
\newblock In A.~I. Bobenko and C.~Klein, editors, {\em Computational Approach
  to {R}iemann Surfaces}, volume 2013 of {\em Lecture Notes in Mathematics},
  pages 3--64, Berlin, 2011. Springer.

\bibitem{A1-BOB-BobenkoMercatSchmies}
A.~I. Bobenko, C.~Mercat, and M.~Schmies.
\newblock Period matrices of polyhedral surfaces.
\newblock In A.~I. Bobenko and C.~Klein, editors, {\em Computational Approach
  to {R}iemann Surfaces}, volume 2013 of {\em Lecture Notes in Mathematics},
  pages 213--226, Berlin, 2011. Springer.

\bibitem{A1-BOB-BobenkoMercatSuris}
A.~I. Bobenko, C.~Mercat, and Y.~B. Suris.
\newblock Linear and nonlinear theories of discrete analytic functions.
  {I}ntegrable structure and isomonodromic {G}reen's function.
\newblock {\em J. Reine Angew. Math.}, 583:117--161, 2005.

\bibitem{A1-BOB-BobenkoPinkallSpringborn}
A.~I. Bobenko, U.~Pinkall, and B.~Springborn.
\newblock Discrete conformal maps and ideal hyperbolic polyhedra.
\newblock Preprint \href{http://arxiv.org/abs/1005.2698v1}{ arXiv:1005.2698v1}.

\bibitem{Bobenko-Springborn-07}
A.I. Bobenko and B.A. Springborn.
\newblock A discrete {L}aplace--{B}eltrami operator for simplicial surfaces.
\newblock {\em Discrete \& Computational Geometry}, 38:740--756, 2007.

\bibitem{A1-BOB-Buecking}
U.~B{\"u}cking.
\newblock Approximation of conformal mappings by circle patterns.
\newblock {\em Geom. Dedicata}, 137:163--197, 2008.

\bibitem{Chelkak-Smirnov-11}
D.~Chelkak and S.~Smirnov.
\newblock Universality in the 2d {I}sing model and conformal invariance of
  fermionic observables.
\newblock {\em Inventiones Mathematicae}, pages 1--66.
\newblock 10.1007/s00222-011-0371-2.

\bibitem{MR2824564}
Dmitry Chelkak and Stanislav Smirnov.
\newblock Discrete complex analysis on isoradial graphs.
\newblock {\em Adv. Math.}, 228(3):1590--1630, 2011.

\bibitem{A1-BOB-CFL_1928}
R.~Courant, K.~Friedrichs, and H.~Lewy.
\newblock {\"U}ber die partiellen {D}ifferentialgleichungen der mathematischen
  {P}hysik.
\newblock {\em Math. Ann.}, 100:32--74, 1928.

\bibitem{A1-BOB-duffin_1953}
R.~J. Duffin.
\newblock Discrete potential theory.
\newblock {\em Duke Math. J.}, 20:233--251, 1953.

\bibitem{A1-BOB-duffin_1959}
R.~J. Duffin.
\newblock Distributed and lumped networks.
\newblock {\em J. Math. Mech.}, 8:793--826, 1959.

\bibitem{A1-BOB-DynnikovNovikov}
I.~A. Dynnikov and S.~P. Novikov.
\newblock Geometry of the triangle equation on two-manifolds.
\newblock {\em Mosc. Math. J.}, 3(2):419--438, 742, 2003.
\newblock Dedicated to Vladimir I. Arnold on the occasion of his 65th birthday.

\bibitem{A1-BOB-ferrand_1944}
J.~Ferrand.
\newblock Fonctions pr\'eharmoniques et fonctions pr\'eholomorphes.
\newblock {\em Bull. Sci. Math. (2)}, 68:152--180, 1944.

\bibitem{A1-BOB-he_schramm_1996}
Zh.-X. He and O.~Schramm.
\newblock On the convergence of circle packings to the {R}iemann map.
\newblock {\em Invent. Math.}, 125(2):285--305, 1996.

\bibitem{A1-BOB-Kenyon_2000}
R.~Kenyon.
\newblock Conformal invariance of domino tiling.
\newblock {\em Ann. Probab.}, 28(2):759--795, 2000.

\bibitem{A1-BOB-lelong-ferrand_1955}
J.~Lelong-Ferrand.
\newblock {\em Repr\'esentation conforme et transformations \`a int\'egrale de
  {D}irichlet born\'ee}.
\newblock Gauthier-Villars, Paris, 1955.

\bibitem{A1-BOB-Mercat_2001}
C.~Mercat.
\newblock Discrete {R}iemann surfaces and the {I}sing model.
\newblock {\em Comm. Math. Phys.}, 218(1):177--216, 2001.

\bibitem{Mercat-08}
C.~Mercat.
\newblock Discrete polynomials and discrete holomorphic approximation.
\newblock arXiv:math-ph/0111043, 2008.

\bibitem{Mercat-02}
Christian Mercat.
\newblock Discrete period matrices and related topics.
\newblock arXiv:math-ph/0206041, 2002.

\bibitem{MR2349680}
Christian Mercat.
\newblock Discrete {R}iemann surfaces.
\newblock In {\em Handbook of {T}eichm\"uller theory. {V}ol. {I}}, volume~11 of
  {\em IRMA Lect. Math. Theor. Phys.}, pages 541--575. Eur. Math. Soc.,
  Z\"urich, 2007.

\bibitem{MR2047000}
M.~Meyer, M.~Desbrun, P.~Schr{\"o}der, and A.~H. Barr.
\newblock Discrete differential-geometry operators for triangulated
  2-manifolds.
\newblock In {\em Visualization and mathematics {III}}, Math. Vis., pages
  35--57. Springer, Berlin, 2003.

\bibitem{MR1246481}
Ulrich Pinkall and Konrad Polthier.
\newblock Computing discrete minimal surfaces and their conjugates.
\newblock {\em Experiment. Math.}, 2(1):15--36, 1993.

\bibitem{A1-BOB-rivin_1994}
I.~Rivin.
\newblock Euclidean structures on simplicial surfaces and hyperbolic volume.
\newblock {\em Ann. of Math. (2)}, 139(3):553--580, 1994.

\bibitem{A1-BOB-rodin_sullivan_1987}
B.~Rodin and D.~Sullivan.
\newblock The convergence of circle packings to the {{R}iemann} mapping.
\newblock {\em Journal of Differential Geometry}, 26(2):349–360, 1987.

\bibitem{A1-BOB-schramm_1997}
O.~Schramm.
\newblock Circle patterns with the combinatorics of the square grid.
\newblock {\em Duke Mathematical Journal}, 86(2):347--389, 1997.

\bibitem{A1-Skopenkov}
M.~Skopenkov.
\newblock The boundary value problem for discrete analytic functions.
\newblock {\em Adv. Math.}, 240:61--87, 2013.

\bibitem{Dorichenko-etal-11}
M.~Skopenkov, M.~Prasolov, and S.~Dorichenko.
\newblock Dissections of a metal rectangle.
\newblock {\em Kvant}, 3:10--16, 2011.

\bibitem{Skopenkov-etal-12}
M.~Skopenkov, V.~Smykalov, and A.~Ustinov.
\newblock Random walks and electric networks.
\newblock {\em Mat. Prosv.}, 16:25--47, 2012.

\bibitem{A1-BOB-smirnov_2001}
S.~Smirnov.
\newblock Critical percolation in the plane: Conformal invariance, {C}ardy's
  formula, scaling limits.
\newblock {\em C. R. Acad. Sci. Paris S\'er. I Math.}, 333(3):239--244, 2001.

\bibitem{A1-BOB-smirnov_2010}
S.~Smirnov.
\newblock Conformal invariance in random cluster models. {I}. {H}olomorphic
  fermions in the {I}sing model.
\newblock {\em Ann. of Math. (2)}, 172(2):1435--1467, 2010.

\bibitem{A1-BOB-SpringbornSchroederPinkall}
B.~Springborn, P.~Schr\"{o}der, and U.~Pinkall.
\newblock Conformal equivalence of triangle meshes.
\newblock In {\em ACM SIGGRAPH 2008 papers}, SIGGRAPH '08, pages 77:1--77:11,
  New York, NY, USA, 2008. ACM.

\bibitem{A1-BOB-Stephenson_Book}
K.~Stephenson.
\newblock {\em Introduction to circle packing}.
\newblock Cambridge University Press, Cambridge, 2005.
\newblock {T}he theory of discrete analytic functions.

\bibitem{A1-BOB-thurston_notes}
W.~P. Thurston.
\newblock The geometry and topology of {Three-Manifolds}.
\newblock Electronic version 1.1, 
  \href{http://www.msri.org/publications/books/gt3m/}{http://www.msri.org/publications/books/gt3m/},
  2002.

\bibitem{Troyanov}
M.~Troyanov.
\newblock Les surface euclidiennes \`a singularit\'es coniques.
\newblock {\em Enseign. Math.}, 32:79--94, 1986.

\bibitem{Wilson-08}
Scott Wilson.
\newblock Conformal cochains.
\newblock {\em Trans. Amer. Math. Soc.}, 360(10):5247--5264, 2008.

\end{thebibliography}

\vskip 0.4cm
\noindent
\textsc{Alexander Bobenko\\
Technische Universit\"at Berlin\\
}

\vskip 0.3cm
\noindent
\textsc{Mikhail Skopenkov\\
Institute for information transmission problems, Russian Academy of Sciences, 
\\
and\\
King Abdullah University of Science and Technology} \\
\texttt{skopenkov@rambler.ru} \quad \url{http://skopenkov.ru}

\end{document}